\documentclass[11pt]{article}\textwidth 160mm\textheight 235mm
\usepackage{amsfonts}
\usepackage{amssymb}
\usepackage{graphicx}
\usepackage{amsmath}
\usepackage{amsthm}
\usepackage{enumitem}
\usepackage{tikz}
\usepackage{mathrsfs}
\usepackage{cancel}
\usepackage{todonotes}
\usetikzlibrary{matrix}
\usetikzlibrary{arrows,shapes}
\usepackage{cite}
\usepackage{hyperref}
\hypersetup{colorlinks=true, urlcolor=blue}
\expandafter\let\expandafter\oldproof\csname\string\proof\endcsname
\let\oldendproof\endproof
\renewenvironment{proof}[1][\proofname]{%
  \oldproof[\ttfamily \scshape \bf #1. ]%
}{\oldendproof}
\def\O{{\cal O}}

\def\B{\mathbb{B}}
\def\R{{\rm I\!R}}
\def\N{{\rm I\!N}}
\def\ox{\bar{x}}
\def\oy{\bar{y}}
\def\oz{\bar{z}}
\def\ov{\bar{v}}

\def\op{\bar{p}}

\def\ss{\scriptsize }

\def\ve{\varepsilon}

\def\emp{\emptyset}
\def\dist{{\rm dist}}

\def\rge{{\rm rge\,}}

\def\Lm{{\Lambda}}
\def\tto{\rightrightarrows}

\def\d{{\rm d}}
\def\sub{\partial}

\def\Hat{\widehat}

\def\Bar{\overline}
\def\ra{\rangle}
\def\la{\langle}
\def\ve{\varepsilon}
\def\omu{\bar{\mu}}

\def\prox {\mbox{\rm prox}\,}

\def\gph{\mbox{\rm gph}\,}
\def\epi{\mbox{\rm epi}\,}
\def\mini{\mbox{\rm minimize}\;\,}

\def\dom{\mbox{\rm dom}\,}

\def\ker{\mbox{\rm ker}\,}

\def\reg{\mbox{\rm reg}\,}

\def\dn{\downarrow}
\def\O{\Omega}

\def\ph{\varphi}

\def\emp{\emptyset}

\def\oR{\Bar{\R}}
\def\lm{\lambda}
\def\olm{\bar\lambda}

\def\dd{\delta}
\def\al{\alpha}
\def\Th{\Theta}

\def\sm{\hbox{${1\over 2}$}}

\def \b{{\}_{k\in\N}}}
\def\L{{\mathscr{L}}}
\def\D{{\mathscr{D}}}
\def\sce{\setcounter{equation}{0}}

\begin{document}
\vspace*{0.5in}
\begin{center}
{\bf PRIMAL SUPERLINEAR CONVERGENCE OF SQP METHODS IN PIECEWISE LINEAR-QUADRATIC COMPOSITE OPTIMIZATION  }\\[1 ex]
 M. EBRAHIM SARABI\footnote{Department of Mathematics, Miami University, Oxford, OH 45065, USA (sarabim@miamioh.edu).}
\end{center}
\vspace*{0.05in}
\small{\bf Abstract.}  This paper  mainly  concerns with the primal superlinear convergence  of the quasi-Newton sequential quadratic programming (SQP) method for piecewise linear-quadratic 
composite optimization problems. We show that  the latter  primal superlinear convergence can be justified under the   noncriticality of Lagrange multipliers and  a version of the Dennis-Mor\'e condition.
Furthermore, we show that if we replace the noncriticality condition with  the second-order sufficient condition, this primal superlinear convergence is equivalent with an appropriate version of the Dennis-Mor\'e condition.
We also recover Bonnans' result in \cite{bo94} for the primal-dual superlinear of the basic SQP method for this class of composite problems under the second-order
sufficient condition and the uniqueness of Lagrange multipliers. To achieve these goals, we first obtain an extension of the reduction lemma for convex 
Piecewise linear-quadratic functions and then  provide a comprehensive analysis of the noncriticality of Lagrange multipliers for composite problems.
We also  establish certain  primal estimates for KKT systems of composite problems, which play a significant  role in our local convergence analysis of the quasi-Newton SQP method. 
 \\[1ex]
\newtheorem{Theorem}{Theorem}[section]
\newtheorem{Proposition}[Theorem]{Proposition}
\newtheorem{Remark}[Theorem]{Remark}
\newtheorem{Lemma}[Theorem]{Lemma}
\newtheorem{Corollary}[Theorem]{Corollary}
\newtheorem{Definition}[Theorem]{Definition}
\newtheorem{Example}[Theorem]{Example}
\newtheorem{Algorithm}[Theorem]{Algorithm}
\renewcommand{\theequation}{{\thesection}.\arabic{equation}}
\renewcommand{\thefootnote}{\fnsymbol{footnote}}

\normalsize

\section{Introduction}\sce
This paper aims to present the local convergence analysis of the sequential quadratic programming  (SQP) methods for the composite optimization problem 
\begin{equation}\label{comp}
\mbox{minimize}\;\;\ph(x)+g(\Phi(x))\quad \mbox{subject to}\;\; x\in \Th,
\end{equation} 
where $\ph:\R^n\to \R$ and  $\Phi:\R^n\to \R^m$ are twice continuously differentiable, and where $\Th$ is 
a polyhedral convex set in $\R^n$ and $g:\R^m\to \oR$ is a convex piecewise linear-quadratic (CPLQ) function.
While the CPLQ function $g$ in \eqref{comp} gives significant flexibility to this problem to cover important classes of optimization problems including classical 
nonlinear programming problems (NLPs), constrained and unconstrained min-max optimization problems, and extended nonlinear programming problems, introduced by Rockafellar in \cite{r97},
the polyhedral convexity of $\Th$ therein makes it possible to  cover nonnegativity constraints, upper and lower bounds on variables, and also situations where 
we want to minimize a function over a linear subspace or an affine subset of $\R^n$. 

While different first- and second-order variational properties of   composite optimization problems have been extensively studied over the last three decades \cite{bs,l01, mmsmor,ms20,r88,r89}, 
  considerable efforts have been made recently toward  developing numerical algorithms, mostly first-order methods,   for this class of problems \cite{bcp,be,dl,lw}. 
  In this work, we present a systematic local convergence analysis  of the SQP methods  for \eqref{comp}. 
Recall that the principal idea of the SQP methods is to solve a sequence of quadratic approximations, called subproblems, whose optimal solutions converge 
under appropriate assumptions to an optimal solution to the original problem. For the composite problem \eqref{comp}, the aforementioned subproblem
at the current primal-dual iterate $(x_k,\lm_k)\in \R^n\times \R^m$  is formulated as 
\begin{equation}\label{subs1}
\begin{cases}
 \mbox{minimize}\;\;\ph(x_k)+ \la \nabla \ph(x_k), x-x_k\ra+ \sm \la H_k(x-x_k),x-x_k\ra+g\big(\Phi(x_k)+\nabla \Phi(x_k)(x-x_k)\big)\\
 \mbox{subject to}\quad  x\in \Th,
 \end{cases}
\end{equation}
where $H_k  $  is an $n\times n$ symmetric matrix for all $k\in \{0\}\cup\N$. In this paper, we study the SQP method for the composite problem \eqref{comp} in which  the matrix $H_k$
satisfies in one of the following conditions: 1) The matrix $H_k$ is of the form
\begin{equation}\label{hk}
H_k= \nabla_{xx}^2L(x_k,\lm_k), \;\; k\in \{0\}\cup\N,
\end{equation}
where $L$ is the Lagrangian associated with \eqref{comp}, defined by \eqref{lag}. When this choice of $H_k$ is utilized in the subproblem \eqref{subs1}, the method 
corresponds to the {\em basic}  SQP method.
2) The matrix $H_k$ is an   approximation of the Hessian matrix $\nabla_{xx}^2L(x_k,\lm_k)$ that satisfies the Dennis-Mor\'e condition 
\begin{equation}\label{dmc}
P_\D\big(\big( \nabla_{xx}^2L(x_k,\lm_k)- H_k\big)(x_{k+1}-x_k)\big)=o(\|x_{k+1}-x_k\|), 
\end{equation}
where the convex cone $\D$ is defined by \eqref{coned} and where $P_\D$ stands for the projection mapping  onto $\D$. When the  latter choice of $H_k$ 
is used in the subproblem \eqref{subs1}, the method corresponds to the {\em quasi-Newton} SQP method. Note that 
 the basic SQP method can be viewed as a natural extension of the Newton method  that is implemented  for the KKT system of 
 the composite problem \eqref{comp}. Indeed, the latter KKT system can be formulated as a generalized equation for which  the Newton method was generalized and studied  
 by Robinson in \cite{rob80};  see \cite[Section~3.1]{is14} for more details.
 
Remember that  given a primal-dual iterate $(x_k,\lm_k)$, the basic SQP method for the composite problem \eqref{comp} is updated to $(x_{k+1},\lm_{k+1})$, where 
 $x_{k+1}$ is  a stationary point of the subproblem \eqref{subs1} with $H_k$ taken from \eqref{hk} and   $\lm_{k+1}$ is  a  Lagrange multiplier associated with  $x_{k+1}$; see Algorithm~\eqref{sqpal}
 for more details.  For   NLPs,
the sharpest   results, established by Bonnans in \cite{bo94}, ensures  the superlinear/quadratic primal-dual convergence for the basic SQP method
under  the second-order sufficient condition and the strict  Mangasarian--Fromovitz constraint qualification -- the latter condition is known to be equivalent to the 
uniqueness of Lagrange multiplier for this class of problems. The pervious results for this framework, obtained by Robinson in \cite{rob74,rob80},  require a stronger version of the second-order sufficient condition
as well as the linear independence constraint qualification both  of which are strictly stronger than the corresponding assumptions, used by Bonnans in \cite{bo94}. Quite recently, 
Burke and Engle \cite[Theorem~7.3]{be} studied the local convergence analysis of the basic SQP method for  the composite 
optimization problem \eqref{comp} with $\Th=\R^n$  and showed that under the strong second-order sufficient condition, the nondegeneracy condition, and the strict complementary condition
this method is superlinear convergent. The approach utilized in \cite{be} is  based on the  local convergence analysis of the Newton method for generalized equations established under the 
strong metric regularity assumption (see \cite[page~194]{dr}) in \cite[Theorem~6D.2]{dr}. In this paper, we take a different path and show that such a primal-dual  superlinear convergence for \eqref{comp}
can be accomplished under some  less restrictive assumptions. Indeed, similar to Bonnans'  result for NLPs, we show that  the second-order sufficient condition and the uniqueness of Lagrange multipliers suffice 
to ensure the primal-dual superlinear convergence of the basic SQP method for \eqref{comp}. 

Our next goal  is  to pursue  conditions that ensure the primal superlinear convergence of  the quasi-Newton SQP method for \eqref{comp}. 
It is important to notice  that the primal superlinear convergence is important when the primal-dual superlinear  convergence is not available, which is the case for 
 the quasi-Newton SQP  method.  In the  local convergence analysis of the quasi-Newton SQP method, the primal-dual convergence is often assumed. Then, the main question is to find   conditions 
 under which the primal superlinear convergence of the method can be achieved. For NLPs, it was observed in \cite{fis12} that  a certain error bound, satisfied  under the second-order sufficient condition,  alone 
 suffices to accomplish  this goal. In particular, it was shown in \cite[Theorem~4.1]{fis12} that  if the second-order sufficient condition holds,   the primal superlinear convergent of the quasi-Newton SQP method for NLPs
 amounts to the Dennis-Mor\'e condition \eqref{dmc} for this class of problems. The interesting fact about this result is that neither the uniqueness of  Lagrange multiplier  nor  any constraint qualification was assumed. 
  This was a remarkable improvement from the previous results that in addition demanded the linear independent constraint qualification; see \cite[Theorem~15.7]{bgls}. Furthermore, it is shown in \cite{fis12} that 
  the second-order sufficient condition can be replaced by the noncriticality of Lagrange multipliers (see Definition~\ref{dcn}), which is less restrictive,  in the latter characterization. 
 In this paper we explore the possibility of  a similar characterization for the primal superlinear convergence of the quasi-Newton SQP method for the composite problem \eqref{comp}
 under the second-order sufficient condition and   the noncriticality assumption. Doing so requires understanding more about the noncriticality of Lagrange multipliers
 of the KKT system of the composite problem \eqref{comp} and obtaining  certain primal error bound estimates for the KKT system of \eqref{comp}. Not only do we achieve these requirements but also we 
 reveal that the proofs of   these results  mainly rely upon two  fundamental properties of the subgradient mappings of CPLQ functions: 1) the reduction lemma and 2) the outer Lipschitzian property.
 It is worth mentioning  that both properties come from the pioneering works of Robinson in \cite{rob1,rob2}. In particular, the reduction lemma,  established   in \cite[Proposition~4.4]{rob2}, 
 tells us that the graph of  the normal cone to a polyhedral convex set, coincides locally  with that of the normal cone to its critical cone; see Theorem~\ref{rlcp} for more details. 
 We will show in Section~\ref{sect02} that a similar observation holds for the subgradient mappings of CPLQ functions.

The rest of the  paper is organized as follows. Section~\ref{sect02} begins with recalling tools of variational analysis utilized throughout the paper and ends with 
a version of the reduction lemma for CPLQ functions. Section~\ref{sect03} presents a characterization of noncriticality of Lagrange multipliers for \eqref{comp}
and explores its relationship with the second-order sufficient condition. Section~\ref{sect04} is devoted to the study of certain primal estimates for the 
KKT system of \eqref{comp} under the second-order sufficient condition and the noncriticality of Lagrange multipliers. Section~\ref{sect05} provides the 
primal-dual superlinear convergence of the basic SQP method for \eqref{comp} under the second-order sufficient condition and the uniqueness of Lagrange multipliers. 
In particular, we show that under the latter conditions the subproblem \eqref{subs1} admits a local optimal solution. Finally, Section~\ref{sect06} establishes 
a characterization of the primal superlinear convergence of the quasi-Newton SQP method for \eqref{comp} via  the Dennis-Mor\'e condition \eqref{dmc}.   

\section{ Preliminary Definitions and Results}\sce \label{sect02}

In this section we first briefly review basic constructions of variational analysis and generalized differentiation employed in the paper;
see \cite{mor18,rw} for more detail. In what follows,   we 
 denote by $\B$  the closed unit ball in the space in question and by $\B_r(x):=x+r\B$ the closed ball centered at $x$ with radius $r>0$. 
 In the  product space $\R^n\times \R^m$, we use the norm $\|(w,u)\|=\sqrt{\|w\|^2+\|u\|^2}$ for any $(w,u)\in \R^n\times \R^m$.
 For any set $C$ in $\R^n$, its indicator function is defined by $\dd_C(x)=0$ for $x\in C$ and $\dd_C(x)=\infty$ otherwise. We denote by $\dist(x,C)$  the distance between $x\in \R^n$ and a set $C$.
 When $C$ is a cone, its polar cone is denoted by $C^*$.
 For a vector $w\in \R^n$, the subspace $\{tw |\, t\in \R\}$ is denoted by $[w]$. 
We write $x(t)=o(t)$ with $x(t)\in \R^n$ and $t>0$ to mean that ${\|x(t)\|}/{t}$ goes to $0$ as $t\dn 0$.
Finally, we denote by $\R_+$ (respectively,  $\R_-$) the set of non-negative (respectively, non-positive) real numbers.

Given a nonempty set $C\subset\R^n$ with $\ox\in C$, the  tangent cone $T_ C(\ox)$ to $C$ at $\ox$ is defined by
\begin{equation*}\label{2.5}
T_C(\ox)=\big\{w\in\R^n|\;\exists\,t_k{\dn}0,\;\;w_k\to w\;\;\mbox{ as }\;k\to\infty\;\;\mbox{with}\;\;\ox+t_kw_k\in C\big\}.
\end{equation*}
We say a tangent vector $w\in T_C(\ox)$ is {derivable} if there exist a constant  $\ve>0$ and an arc $\xi:[0,\ve]\to C$ such that  $\xi(0)=\ox$ and $\xi'_+(0)=w$, where $\xi'_+$ signifies the right derivative of $\xi$
at $0$, defined by 
$$
\xi'_+(0):=\lim_{t\dn 0}\frac{\xi(t)-\xi(0)}{t}.
$$
The set $C$ is called geometrically derivable at $\ox$ if every tangent vector $w$ to $C$ at $\ox$ is derivable.  Convex sets are 
important examples of geometrically derivable sets. The (Mordukhovich/limiting) {normal cone} to $C$ at $\ox\in C$ is given by
\begin{eqnarray*}\label{2.4}
N_C(\ox)=\big\{v\in\R^n\;\big|\;\exists\,x_k{\to}\ox,\;\;v_k\to v\;\;\mbox{with}\;\;x_k\in C,\;\;v_k\in \Hat N_C(x_k)\big\},
\end{eqnarray*}
where $\Hat N_C(x)=\big\{w\in\R^n\;\big|\; \la w,u-x\ra\le o(\|u-x\|)\;\;\mbox{for}\;u\in C\big\}$ is the regular normal cone to $C$ at $x$.
When $C$ is convex, both normal cones reduce to the normal cone in the sense of convex analysis.
Given the function $f:\R^n \to \oR:= (-\infty, \infty]$, its domain and epigraph are defined, respectively, by 
$$
\dom f =\big\{ x \in \R^n |\; f(x) < \infty \big \}\quad \mbox{and}\quad \epi f=\big \{(x,\al)\in \R^n\times \R|\, f(x)\le \al\big \}.
$$
When $f:\R^n \to \oR$ is  finite at $\ox$, the  (limiting) subdifferential  of $f$ at $\ox$   is defined  by 
\begin{equation*}\label{sub}
\partial f(\ox):=\big\{v\in\R^n\;\big|\;(v,-1)\in N_{\scriptsize{\epi f}}\big(\ox,f(\ox)\big)\big\},
\end{equation*}
which  reduces to the classical subgradient set of $f$ in the sense of convex analysis when $f$ is convex.
Similarly, one can define the regular subdifferential of $f$ at $\ox$, denoted by $\Hat \sub f(\ox)$, by replacing 
the   normal cone $N_{\scriptsize{\epi f}}\big(\ox,f(\ox)\big)$ in the above definition by $\Hat N_{\scriptsize{\epi f}}\big(\ox,f(\ox)\big)$.

Consider a set-valued mapping $F:\R^n\tto\R^m $ with its domain and graph  defined, respectively, by 
$$ \dom F=\big\{x\in\R^n |\;F(x)\ne\emp\big\}\quad \mbox{and}\quad \gph F=\big\{(x,y)\in\R^n\times\R^m |\;y\in F(x)\big\}.$$
 The {graphical derivative} of $F$ at $(\ox,\oy)\in \gph F$ is defined by 
\begin{equation*}\label{gder}
DF(\ox , \oy)(w)=\big\{u\in\R^m|\;(w,u)\in T_{\scriptsize{\gph F}}(\ox,\oy)\big\},\quad w\in\R^n.
\end{equation*}
The set-valued mapping  $F$ is called {proto-differentiable} at $\ox$ for $\oy$   if the set $\gph F$ is geometrically derivable at $(\ox,\oy)$. When this condition holds for $F$, we refer to $DF(\ox,\oy)$ as the {proto-derivative} of $F$ at $\ox$ for $\oy$.

Recall that a set-valued mapping $F:\R^n\tto\R^m$ is calm at  $(\ox,\oy)\in\gph F$ if there are a constant  $\ell \in \R_+$  and neighborhoods $U$ of $\ox$ and $V$ of $\oy$ so that 
\begin{equation*}\label{metreq}
F(x)\cap V\subset F(\ox)+\ell\|x-\ox\|\B\quad \;\;\mbox{for all}\; x\in U.
\end{equation*}
 The  set-valued mapping $F$ is called isolated calm  at $(\ox,\oy)$ if there are
 a constant $\ell\in \R_+$ and a neighborhood $U$ of $\ox$ and neighborhood $U$ of $\ox$  such that the inclusion  
 $$
 F(x)\cap V\subset \{\oy\}+\ell\|x-\ox\|\B\;\;\mbox{for all}\; x\in U
 $$
holds. It is known that these calmness properties amount to the following metric subregularity properties for inverse mappings, respectively. A set-valued $F$
is called metrically subregular at $(\ox,\oy)$ if there are a constant $\ell\in \R_+$ and neighborhood $U$ of $\ox$  such that 
$$
\dist\big(x,F^{-1}(y)\big)\le \ell\,\dist\big(y,F(\ox)\big)\quad \;\;\mbox{for all}\; x\in U.
$$
It is called strongly metrically subregular at this point if 
there are a constant $\ell\in \R_+$ and and neighborhood $U$ of $\ox$ such that 
$$
\|x-\ox\| \le \ell\, \dist\big(y,F(\ox)\big)\quad \;\;\mbox{for all}\; x\in U.
$$

Given  a function  $f:\R^n \to \oR$ and  a point $\ox$ with $f(\ox)$ finite, the subderivative function $\d f(\ox)\colon\R^n\to[-\infty,\infty]$ is defined by
$$
{\mathrm d}f(\ox)(w)=\liminf_{\substack{
   t\dn 0 \\
  u\to w
  }} {\frac{f(\ox+tu)-f(\ox)}{t}},\quad w\in \R^n.
$$
 The critical cone of $f$ at $\ox$ for $\ov$ with $(\ox,\ov)\in \gph \sub f$ is defined by 
$$
{K_f}(\ox,\ov):=\big\{w\in \R^n\big|\;\la\bar v,w\ra=\d f(\ox)(w)\big\}.
$$
When $f=\dd_C$, where $C$ is a nonempty subset of $\R^n$, the critical cone of $\dd_C$ at $\ox$ for $\ov$ is denoted by $K_C(\ox,\ov)$. In this case, the above definition of the critical cone of a function 
boils down to  the well known concept of a critical cone of a set (see \cite[page~109]{dr}), namely $K_C(\ox,\ov)=T_C(\ox)\cap [\ov]^\perp$ because of $\d \dd_C(\ox)=\dd_{T_C(\ox)}$.
Define the parametric  family of 
second-order difference quotients for $f$ at $\ox$ for $\ov\in \R^n$ by 
\begin{equation*}\label{lk01}
\Delta_t^2 f(\bar x , \ov)(w)=\dfrac{f(\ox+tw)-f(\ox)-t\langle \ov,\,w\rangle}{\frac{1}{2}t^2}\quad\quad\mbox{with}\;\;w\in \R^n, \;\;t>0.
\end{equation*}
If $f(\ox)$ is finite, then the {second subderivative} of $f$ at $\ox$ for $\ov$   is given by 
\begin{equation*}\label{ssd}
\d^2 f(\bar x , \ov)(w)= \liminf_{\substack{
   t\dn 0 \\
  w'\to w
  }} \Delta_t^2 f(\ox , \ov)(w'),\;\; w\in \R^n.
\end{equation*}

Following \cite[Definition~13.6]{rw}, a function $f:\R^n \to \oR$ is said to be {twice epi-differentiable} at $\bar x$ for $\ov\in\R^n$, with $f(\ox) $ finite, 
if for every sequence $t_k\downarrow 0$ and every $w\in\R^n$, there exists a sequence $w_k \to w$ such that
\begin{equation*}\label{dedf}
\d^2 f(\bar x,\ov)(w) = \lim_{k \to \infty} \Delta_{t_k}^2 f(\ox , \ov)(w_k).
\end{equation*}

 Recall  that a function $g:\R^m\to \oR$ is called  piecewise linear-quadratic if $\dom g = \cup_{i=1}^{s} C_i$, where  $s\in \N$ and $C_i $ are polyhedral convex  sets for all  $i = 1, \ldots, s$, and if $g$ has a representation of the form
\begin{equation} \label{PWLQ}
g(z) = \sm \langle A_i z ,z \rangle + \langle a_i ,z \rangle + \alpha_i  \quad \mbox{for all} \quad  z \in C_i,
\end{equation}
where $A_i$ is an $m \times m$ symmetric matrix, $a_i\in \R^m$, and $\alpha_i\in \R$ for $i = 1, \ldots, s$. Take $\oz\in \dom g$ and define the {\em active} indices  of the domain of $g$ at $\oz$ by
\begin{equation}\label{act}
I(\oz)=\big \{i\in \{1,\ldots, s\} |\, \oz\in C_i\big \}.
\end{equation}
When such a function is convex, 
it acquires   remarkable first- and second-order variational properties as reported below. The first part of the following result comes  from \cite[page~487]{rw} and the second part is taken from \cite[Proposition~10.21]{rw}.

\begin{Proposition}[first-order variational properties of CPLQ] \label{fop}Assume that $g:\R^m\to \oR$ is a CPLQ function with the representation \eqref{PWLQ}
and that $\oz\in \dom g$.
Then the following conditions hold:
 \begin{itemize}[noitemsep,topsep=0pt]
 \item [{\rm (a)}]  the subdifferential of $g$ at $\oz$ can be calculated by 
\begin{equation}\label{sub}
\sub g(\oz)=\bigcap_{i\in I(\oz)} \big \{v\in \R^m |\, v-A_i \oz- a_i\in N_{ C_i}(\oz)\big \};
\end{equation}

 \item [{\rm (b)}] the domain of the subderivative of $g$ at $\oz$ can be calculated by $\dom \d g(\oz)=T_{\scriptsize{\dom g}}(\oz)=\bigcup_{i\in I(\oz)} T_{\ss C_i}(\oz)$. 
 Moreover , if $w\in T_{\ss C_i}(\oz)$ for some $i\in I(\oz)$, then we have $\d g(\ox)(w)= \la A_i \oz+ a_i, w\ra$.
 \end{itemize}
\end{Proposition}

Next, we recall second-order variational properties of CPLQ functions. 

\begin{Proposition}[second-order variational properties of CPLQ] \label{sop} Assume that $g:\R^m\to \oR$ is a CPLQ function with the representation \eqref{PWLQ}
and that $\oz\in \dom g$ and $\ov\in \sub g(\oz)$. Set $\ov_i:=\ov -A_i \oz -a_i$ for all $i\in I(\oz)$. 
Then the following conditions hold:
 \begin{itemize}[noitemsep,topsep=0pt]
 \item [{\rm (a)}]   the critical cone of $g$ at $\oz$ for $\ov$ has a representation of the form 
\begin{equation}\label{cc2}
{  K_g}(\oz,\ov)=\bigcup_{i\in I(\oz)} {K}_{C_i}(\oz,\ov_i)\quad \mbox{with}\;\;  {K}_{C_i}(\oz,\ov_i)= T_{ C_i}(\oz) \cap [\ov_i]^{\bot};
\end{equation}
 \item [{\rm (b)}] the function $g$ is twice epi-differentiable at $\oz$ for $\ov$ and its second subderivative  at this point can be calculated by 
\begin{equation}\label{pwfor}
 \d^2 g(\oz , \ov ) (w)  
=\begin{cases}
\la A_i w, w\ra&\mbox{if}\;\; w\in {K}_{C_i}(\oz,\ov_i),\\
\infty&\mbox{otherwise};
\end{cases}
\end{equation}
  \item [{\rm (c)}] the subgradient mapping $\sub g$ is proto-differenitable at $\oz$ for $\ov$ and  $\dom D(\sub g)(\oz,\ov)=K_g(\oz,\ov)$. Moreover,  for any $w\in K_g(\oz,\ov)$, the proto-derivative of $\sub g$ at $\oz$ for $\ov$ can be calculated by 
\begin{equation}\label{gdpw}
D(\sub g)(\oz,\ov)(w)=\bigcap_{i\in{\mathfrak J}(w)}\big\{u\in\R^m\big|\;u-A_i w\in N_{\ss K_{C_i}(\oz,\ov_i)}(w)\big\},
\end{equation}
 where ${\mathfrak J}(w):=\big\{i\in I(\oz)\big|\;w\in K_{C_i}(\oz,\ov_i)\big\}$. In particular, we have $D(\sub g)(\oz,\ov)(0)=K_g(\oz,\ov)^*$.
  \end{itemize}
\end{Proposition}
\begin{proof} Part (a) follows immediately from the definition of the critical cone of $g$ at $\oz$ for $\ov$ together with Proposition~\ref{fop}(b).
Part (b) is taken from  \cite[Proposition~13.9]{rw}. The claimed proto-differentiability of $\sub g$ in (c) results from (b) and \cite[Theorem~13.40]{rw}.   The proto-derivative \eqref{gdpw}   comes from  \cite[Proposition~7.3]{mmsmor}. 
The given formula for $D(\sub g)(\oz,\ov)(0)$ was justified in \cite[Theorem~8.1]{mmsmor}. 
\end{proof}

 We proceed   by providing an extension of the    reduction lemma for CPLQ functions. Recall from \cite[Lemma~2E.4]{dr} that  for polyhedral convex sets this result can be stated as follows: If  ${ \Th}$ is  a polyhedral  convex set  in $\R^n$  and  $(\ox,\oy)\in \gph N_\Th$,
 then  there exists a neighborhood ${\cal O}$ of $(0,0)\in \R^m\times \R^m$
for which we have
\begin{equation}\label{rl10}
\big ( (\gph N_\Th) - (\ox,\oy)\big )\cap {\cal O}=\big (\gph N_{  K_\Th(\ox,\oy)}\big )\cap {\cal O}.
\end{equation}
The reduction lemma was appeared first in \cite[Proposition~4.4]{rob2} and has played an important role in sensitivity  analysis of optimization problems with polyhedral structures.
 It is important to notice that $N_{K_\Th(\ox,\oy)}$ appearing on the right-hand side of this equality is, indeed, the proto-derivative of $N_\Th$, namely 
 \begin{equation}\label{pdp}
 DN_\Th(\ox,\oy)=N_{K_\Th(\ox,\oy)};
 \end{equation} 
 see \cite[equation~(9.6)]{mmstams} or \cite[Example~4A.4]{dr} for a proof of this result. Using this observation, we show below that  a similar result  can be justified for CPLQ functions.
 \begin{Theorem}[reduction lemma for  CPLQ functions]\label{rlcp} Let  $g:\R^m\to \oR$  be a CPLQ function and $(\oz,\ov)\in \gph \sub g$. 
 Then there exists a neighborhood ${\cal O}$ of $(0,0)\in \R^m\times \R^m$
for which we have
\begin{equation}\label{rl2}
\big ( (\gph \sub g) - (\oz,\ov)\big )\cap {\cal O}=\big (\gph D(\sub g)(\oz,\ov)\big )\cap {\cal O}.
\end{equation}
\end{Theorem}
\begin{proof} Since  $(\oz,\ov)\in \gph \sub g$, we deduce from \eqref{sub} that 
$$
\ov_i=\ov -A_i \oz -a_i\in N_{ C_i}(\oz)\quad \mbox{for all}\; i\in I(\oz).
$$
We know from  \eqref{PWLQ} that for any $i=1,\ldots,s$, the set $C_i$ is a polyhedral convex set. By \eqref{rl10}, we find a neighborhood ${\cal O}_i$ of $(0,0)\in \R^m\times \R^m$
such that 
\begin{equation}\label{rl1}
\big ( (\gph N_{ C_i}) - (\oz,\ov_i)\big )\cap {\cal O}_i=\big (\gph N_{  K_{ C_i}(\oz,\ov_i)}\big )\cap {\cal O}_i\quad \mbox{for all}\; i=1,\ldots,s.
\end{equation}
Pick $\ve>0$ such that $I(z)\subset I(\oz)$ for all $z\in \B_\ve(\oz)$ and that 
\begin{equation}\label{exact}
T_{C_i}(\oz)\cap \B_\ve(0)=(C_i-\oz)\cap \B_\ve(0)\quad  \mbox{for all}\; i\in I(\oz).
\end{equation}
Indeed, the latter follows directly from \cite[Exercise~6.47]{rw} since $C_i$ are polyhedral convex sets. Shrinking $\ve$ if necessary, assume without loss of generality that ${\cal O}:=\B_{\ve/2\al}(0,0)\subset \cap_{i=1}^s {\cal O}_i$,
where $\al=\max_{i\in I(\oz)}\{1,\|A_i\| \}$.  To justify \eqref{rl2}, let $(z,v)\in \big ( \gph \sub g\big )\cap \big((\oz,\ov)+{\cal O}\big)$. We are going to show that $(z-\oz,v-\ov)\in \gph D(\sub g)(\oz,\ov)$. 
According to \eqref{gdpw}, this can be justified by showing that for any $i\in {\mathfrak J}(x-\ox)$ we have $v-\ov-A_i (z-\oz)\in N_{\ss K_{C_i}(\oz,\ov_i)}(z-\oz)$.
So pick $i\in {\mathfrak J}(z-\oz)$. By definition, this tells us that $z-\oz\in K_{C_i}(\oz,\ov_i)\subset T_{C_i}(\oz)$, which by \eqref{exact} yields $z\in C_i$, meaning that $i\in I(z)$. 
Using this,  $(z,v)\in  \gph \sub g$, and \eqref{sub} confirms  that $(z,v-A_iz-a_i)\in \gph N_{C_i}$. Observe that 
\begin{eqnarray*}
\|(z,v-A_iz-a_i)- (\oz,\ov_i)\|&=&\sqrt{\|z-\oz\|^2+\|v-\ov-A_i(z-\oz)\|^2}\\
&\le &\sqrt{3}\al \sqrt{\|z-\oz\|^2+\|v-\ov\|^2}<\ve.
\end{eqnarray*}
This, combined with \eqref{rl1}, indicates that $v-\ov-A_i (z-\oz)\in N_{\ss K_{C_i}(\oz,\ov_i)}(z-\oz)$ and so we arrive at the inclusion `$\subset$' in \eqref{rl2}.

Turning now to verify  the opposite inclusion `$\supset$' in \eqref{rl2}, pick $(w,u)\in \big (\gph D(\sub g)(\oz,\ov)\big )\cap {\cal O}$. 
We are going to show that $(\oz+w,\ov+u)\in \gph \sub g$. 
The latter inclusion via \eqref{sub} amounts to showing that for any $i\in I(\oz+w)$, we have 
$$u+\ov-A_i(\oz+w)-a_i\in N_{C_i}(\oz+w).$$
To prove this, pick $i\in I(\oz+w)$, meaning that $\oz+w\in C_i$. By the definition of ${\cal O}$ and \eqref{exact}, we obtain $w\in T_{C_i}(\oz)$.
Moreover, $(w,u)\in \gph D(\sub g)(\oz,\ov)$ and Proposition~\ref{sop}(c) result in $w\in \dom D(\sub g)(\oz,\ov)=K_g(\oz,\ov)$. This inclusion, $w\in T_{C_i}(\oz)$, and Proposition~\ref{fop}(b) bring us to 
$$
\la w,\ov\ra= \d g(\oz)(w)= \la A_i \oz+ a_i, w\ra,
$$
which in turn yields $\la w,\ov_i\ra=0$. So we get $w\in  T_{C_i}(\oz)\cap [\ov_i]^\perp=K_{C_i}(\oz,\ov)$. Since $I(\oz+w)\subset I(\oz)$ by the definition of ${\cal O}$, we arrive at $i\in {\mathfrak J}(w) $,
where the index set ${\mathfrak J}(w)$ is defined in Proposition~\ref{sop}(c). Thus, by the latter proposition, we obtain $(w, u-A_iw)\in\gph  N_{\ss K_{C_i}(\oz,\ov_i)}$. 
Since 
\begin{eqnarray*}
\| (w, u-A_iw)\|=\sqrt{\|w\|^2+\|u-A_iw\|^2}\le \sqrt{3}\al \sqrt{\|u\|^2+\|w\|^2}<\ve,
\end{eqnarray*}
we get $ (w, u-A_iw)\in \B_\ve(0,0)\subset {\cal O}_i$. Appealing now to \eqref{rl1} implies that 
$(w, u-A_iw)+(\oz,\ov_i)\in \gph N_{C_i}$. This confirms that $u+\ov-A_i(\oz+w)-a_i\in N_{C_i}(\oz+w)$ and thus justifies the inclusion `$\supset$' in \eqref{rl2}.
\end{proof}

We continue by showing that the proto-derivative of the subgradient mapping of a CPLQ function enjoys  the outer/upper Lipschitzian property.
\begin{Proposition}[outer Lipschitzian of proto-derivative]\label{ouli} Assume that $g:\R^m\to \oR$ is a CPLQ function 
and that $\oz\in \dom g$ and $\ov\in \sub g(\oz)$. Then the following conditions hold:
 \begin{itemize}[noitemsep,topsep=0pt]
 \item [{\rm (a)}] there are a neighborhood $U$ of $\oz$ and a constant $\ell\ge 0$ such that 
 $$
 \sub g(z)\subset \sub g(\oz)+\ell \|x-\ox\|\B\quad \mbox{for all}\;z\in U;
 $$
  \item [{\rm (b)}] for any $w\in K_g(\oz,\ov)$, there are a neighborhood $W$ of $w$ and a constant $\ell\ge 0$ such that 
  $$
D(\sub g)(\oz,\ov)(u)\subset D(\sub g)(\oz,\ov)(w)+\ell\|u -w\|\B  \quad \mbox{for all}\;u\in W.
  $$
 \end{itemize}
\end{Proposition}
\begin{proof} Since $g$ is CPLQ, it follows from the proof of \cite[Theorem~11.14(b)]{rw} that $\gph\partial g$ is a union of finitely many  polyhedral convex  sets. 
By assumptions, we have  $\oz\in \dom \sub g$.
These together with \cite[Proposition~1]{rob1} (see also \cite[Theorem~3D.1]{dr})
proves (a). 

To verify (b), observe first by Proposition~\ref{sop}(b) that $g$ is twice epi-differentiable at $\oz$ for $\ov$. 
Appealing to \cite[Theorem~13.40]{rw} tells us that 
\begin{equation}\label{gdpl}
D(\sub g)(\oz,\ov)(w)=\sub\big(\sm  \d^2 g(\oz , \ov )\big) (w).
\end{equation}
According to Proposition~\ref{sop}(b), the function $\sm  \d^2 g(\oz , \ov )$ is CPLQ. Employing again \cite[Theorem~11.14(b)]{rw} shows that $\gph \sub\big(\sm  \d^2 g(\oz , \ov )\big) $
 is a union of finitely many  polyhedral convex  sets. Since $w\in K_g(\oz,\ov)=\dom D(\sub g)(\oz,\ov)$, Robinson's observation in  \cite[Proposition~1]{rob1}, combined with \eqref{gdpl},  justifies (b). 
 \end{proof}
 
 Recall that the Lagrangian of \eqref{comp}
is given by 
\begin{equation}\label{lag}
L(x,\lm)= \ph(x)+\la \Phi(x),\lm\ra, \quad \mbox{with}\;\;(x,\lm)\in \R^n\times \R^m.
\end{equation}
Note that a slightly  different Lagrangian has been utilized--see for instance \cite{be}--for the composite problem \eqref{comp} by subtracting  the Fenchel conjugate function $g^*(\lm)$ from the Lagrangian above.
We, however, do not consider such a term in the Lagrangian for \eqref{comp} since it does not have any impacts on second-order analysis conducted in this paper. 
The   Karush-Kuhn-Tucker (KKT)   system associated with the composite \eqref{comp} is given by 
\begin{equation}\label{vs} 
0\in \nabla_x L(x,\lm)+N_\Th(x),\;\;\lm\in \sub g(\Phi(x)),
\end{equation}
where $\nabla_x L(x,\lm)= \nabla \ph(x)+\nabla \Phi(x)^*\lm$ with  $\nabla \Phi(\ox)^*$ standing for the transpose of the  Jacobian matrix $\nabla \Phi(\ox)$. 
Given a point $\ox\in\R^n$, we define the set of {  Lagrange multipliers} of the KKT system \eqref{vs} associated with $\ox$ by
\begin{equation}\label{laset}
\Lambda(\ox):=\big\{\lm\in\R^m\;\big|\; 0\in \nabla_x L(\ox,\lm)+N_\Th(\ox),\;\lm\in \sub g(\Phi(\ox))\big\}.
\end{equation}
If  $(\ox,\olm)$ is a solution to the KKT  system \eqref{vs}, then we get  $\olm\in\Lambda(\ox)$.
If  $\lm\in\Lambda(\ox)$,   we can conclude that 
\begin{eqnarray}\label{stat}
0\in \nabla \ph(\ox)+ \nabla\Phi(\ox)^*\lm+ N_\Th(\ox)&\subset& \nabla \ph(\ox)+ \nabla\Phi(\ox)^*\sub g\big(\Phi(\ox)\big)+ N_\Th(\ox)\nonumber \\
&\subset&  \nabla \ph(\ox)+  \Hat\sub (g\circ \Phi)(\ox)+ N_\Th(\ox)\nonumber \\
&\subset&  \Hat \sub\big(\ph+g\circ \Phi +\dd_\Th\big)(\ox) \nonumber\\
&\subset&    \sub\big(\ph+g\circ \Phi +\dd_\Th\big)(\ox),
\end{eqnarray}
where both second and third inclusions come from  \cite[Exercise~10.7]{rw} and \cite[Corollary~10.9]{rw}, respectively. 

We end this section by recalling second-order optimality conditions for the composite problem \eqref{comp}, which are taken from 
 \cite[Exercise~13.26]{rw}. We should add here that the latter result was written in \cite{rw} for a subclass of \eqref{comp} for which the CPLQ function $g$  in \eqref{comp} has the  representation \eqref{theta}.
 It is rather easy to see that this result  holds for any CPLQ functions. Below we provide a proof, which is in fact an elaboration of  the proof of \cite[Exercise~13.26]{rw}.  
 
 \begin{Proposition}[second-order optimality conditions]\label{sooc} Assume that $\ox\in \Th$,  $\Phi(\ox)\in \dom g$, and $\olm\in \Lm(\ox)$, where $\Th$, $\Phi $, and $g$ are taken from  \eqref{comp},
 and that the basic constraint qualification 
 \begin{equation}\label{rcq}
 -\nabla \Phi(\ox)^* u\in N_\Th(\ox), \;\; u\in N_{\ss\dom g}(\Phi(\ox))\implies u=0
 \end{equation}
 holds. Then the following second-order optimality conditions hold:
  \begin{itemize}[noitemsep,topsep=0pt]
 \item [{\rm (a)}] if $\ox$ is a local minimizer of \eqref{comp},    then the second-order necessary condition
\begin{equation*}
\max_{\lm\in\Lambda(\ox)}\big\{\langle\nabla_{xx}^2L(\bar x,\lm)w,w\rangle+\d^2 g(\Phi(\bar x),\lm ) (\nabla \Phi(\ox)w )\big\}  \ge 0
\end{equation*}
is satisfied for all vectors $w\in \D$, where the convex cone $\D$ is defined by 
 \begin{equation}\label{coned}
 \D:=K_\Th\big(\ox,-\nabla_x L(\ox,\olm)\big)\cap \big\{w\in \R^n\big|\; \nabla \Phi(\ox)w\in K_g(\Phi(\ox),\olm)\big\}.
 \end{equation}
 \item [{\rm (b)}]  the  second-order condition
\begin{equation*}\label{sscc}
\max_{\lm\in\Lambda(\ox)}\big\{\langle\nabla_{xx}^2L(\bar x,\lm)w,w\rangle+\d^2 g (\Phi(\bar x),\lm ) (\nabla \Phi(\ox)w )\big\}>0 \quad\mbox{for all}\;\; w\in \D\setminus\{0\}
\end{equation*}
amounts to the existence of positive constants $\ell$ and $\ve$ such that the quadratic growth condition
\begin{equation*}\label{quadg2}
\ph(x)+g (\Phi(x))\ge \ph(\ox)+g( \Phi(\ox))+\frac{\ell}{2}\|x-\ox\|^2\;\mbox{ for all }\;x\in\B_{\ve}(\ox)\cap \Th
\end{equation*}
holds.
 \end{itemize}
 \end{Proposition} 
  \begin{proof} We begin the proof by showing that for any $\lm\in \Lm(\ox)$, we have $\D=\D_{\lm}$, where $\D_\lm$ is defined by replacing $\olm$ with $\lm$ in the definition of the convex cone $\D$ in \eqref{coned}. 
  To justify it, let $w\in \D$ and so  conclude  that $\nabla \Phi(\ox)w\in K_g(\Phi(\ox),\olm)$, $w\in T_\Th(\ox)$, and $\la w,\nabla_x L(\ox,\olm)\ra=0$,
 which in turn yield   
\begin{equation}\label{bv}
 \d g (\Phi(\ox))(\nabla \Phi(\ox)w)=\la \olm,\nabla\Phi(\ox)w\ra=\la \nabla\Phi(\ox)^*\olm,w\ra= \la -\nabla\ph(\ox),w\ra.
\end{equation}
 Since $\lm\in \Lm(\ox)$, we get $-\nabla_x L(\ox,\lm)\in N_\Th(\ox)$. This together with $w\in T_\Th(\ox)$ implies that $\la w,\nabla_x L(\ox,\lm) \ra\ge 0$, and so we obtain 
$$
 \la\lm, \nabla\Phi(\ox)w\ra \ge  \la -\nabla\ph(\ox),w\ra.
$$
 Combining these results in $\d g (\Phi(\ox))(\nabla \Phi(\ox)w)\le \la \lm, \nabla\Phi(\ox)w\ra$. Since the opposite inequality always holds due to $\lm\in \sub g(\Phi(\ox))$ (cf. \cite[Exercise~8.4]{rw}),
 we arrive at 
\begin{equation}\label{bv2}
\d g (\Phi(\ox))(\nabla \Phi(\ox)w)=\la \lm,\nabla\Phi(\ox)w\ra= \la \nabla\Phi(\ox)^*\lm,w\ra,
\end{equation}
 which yields $\nabla\Phi(\ox)w\in K_g(\Phi(\ox),\lm)$. Moreover, by \eqref{bv}-\eqref{bv2}, we obtain $\la \nabla\Phi(\ox)^*\lm,w\ra=\la \nabla\Phi(\ox)^*\olm,w\ra$. The latter equality and $\la w,\nabla_x L(\ox,\olm)\ra=0$ results in $\la w,\nabla_x L(\ox,\lm)\ra=0$,
 meaning  that $w\in K_\Th \big(\ox,\nabla_x L(\ox,\lm)\big )$. This shows that $w\in \D_{\lm}$. The opposite inclusion can be justified similarly. 
 
Set $f:=\ph+g\circ \Phi $. It follows from $\olm\in \Lm(\ox)$ and \eqref{stat} that $0\in \sub\big(f+\dd_\Th\big)(\ox)$ and $-\nabla_x L(\ox,\olm)\in N_\Th(\ox)$. 
 By \cite[equation~(3.10)]{mmstams} (see also \cite[Example~13.17]{rw}), we have 
\begin{equation}\label{ssp}
\d^2 \dd_\Th(\ox,-\nabla_x L(\ox,\olm))=\dd_{K_\Th(\ox,-\nabla_x L(\ox,\olm))}.
\end{equation}
Using this and \cite[Theorem~3.4]{pr92}, we conclude for every $w\in \R^n$ that 
 \begin{equation}\label{ssca2}
 \d^2 \big(f+\dd_\Th\big) (\ox,0)(w)=\max_{\lm\in \Lm(\ox)}\Big\{\dd_{K_\Th (\ox,-\nabla_x L(\ox,\lm) )}(w)+ \d^2 g(\Phi(\bar x),\lm ) (\nabla \Phi(\ox)w )+\langle\nabla_{xx}^2L(\bar x,\lm)w,w\rangle\Big\}.
 \end{equation}
 For any $w\in \R^n$,  we claim   that 
\begin{equation}\label{ssca}
  \d^2 \big(f+\dd_\Th\big) (\ox,0)(w)=\max_{\lm\in \Lm(\ox)}\Big\{ \d^2 g(\Phi(\bar x),\lm ) (\nabla \Phi(\ox)w )+\langle\nabla_{xx}^2L(\bar x,\lm)w,w\rangle\Big\}+\dd_{K_\Th (\ox,-\nabla_x L(\ox,\olm) )}(w).
\end{equation}
 Indeed, if $w\notin \D$, both sides of \eqref{ssca} equal $\infty$ due to   \eqref{pwfor}, \eqref{ssp}, and \eqref{ssca2}. If $w\in \D$, one can see that 
 the right-hand sides in \eqref{ssca2} and \eqref{ssca} coincide since $w\in \D_\lm$ for every  $\lm\in \Lm(\ox)$. By \eqref{ssca}, both claims in (a) and (b) 
 fall immediately  out of \cite[Theorem~13.24]{rw}. 
 \end{proof}
 
Note that while it may seem that  the definition of the convex cone $\D$ from \eqref{coned}  depends on $\olm$, the above proof reveals that 
it  will not change if we replace $\olm$ with any other Lagrange multiplier associated with $\ox$.

 \begin{Remark}[equivalent form of the composite problem]\label{remeq}{\rm Note that the composite optimization problem \eqref{comp} 
 can be equivalently expressed as 
 \begin{equation}\label{comp3}
\mini \ph(x)+\psi \big(x, \Phi(x)\big)\quad \mbox{subject to}\;\; x\in \R^n,
\end{equation} 
where $\psi:\R^n\times \R^m\to \oR$ is defined by $\psi(x,y)=\dd_\Th(x)+g(y)$  and 
where $\ph$, $g$, $\Phi $, and $\Th$ are taken from \eqref{comp}. According to \cite[Exercise~10.22(a)]{rw}, $\psi$ is a CPLQ function. 
So one can assume without loss of generality that $\Th=\R^n$ in \eqref{comp}.  The downside of this reduction is that one should write the final results in terms of the initial data and 
this requires a sum rule for different second-order constructions,  utilized in this paper. While this is not hard to achieve, it requires some effort. Since such a set $\Th$ appears 
in important applications of \eqref{comp} such as extended linear-quadratic programming problems (see Example~\ref{elqp}), we will proceed with \eqref{comp} in this paper.
}
 \end{Remark}

\section{ Characterizations of Noncriticality of Lagrange Multipliers}\sce  \label{sect03} 
In this section, we aim to present characterizations of noncritical multipliers of the KKT  system associated with the composite optimization problem \eqref{comp}.
To this end, we begin by  introducing  the concepts of critical and noncritical multipliers  for the KKT system \eqref{vs}. 

\begin{Definition}[critical and noncritical multipliers]\label{dcn}
Let $(\ox,\olm)$ be a solution to the KKT system \eqref{vs}. Then the multiplier $\olm\in\Lambda(\ox)$ is said to be {  critical} for \eqref{vs} if there is a nonzero vector $w\in\R^n$  satisfying the inclusion
\begin{equation}\label{crc}
0\in\nabla^2_{xx}L (\ox,\olm)w+\nabla \Phi(\ox)^*D(\sub g) (\Phi(\ox),\olm ) (\nabla \Phi(\ox)w ) +DN_\Th(\ox, -\nabla_x L(\ox,\olm))(w).
\end{equation}
The multiplier $\olm\in\Lambda(\ox)$ is {  noncritical} for \eqref{vs} if \eqref{crc} admits only the trivial solution $w=0$. 
\end{Definition}

If the polyhedral convex set $\Th=\R^n$, then Definition~\ref{dcn} clearly boils down to  \cite[Definition~3.1]{ms17}. 
The concepts of critical and noncritical multipliers were introduced  by
Izmailov in \cite{iz05} for the KKT system \eqref{vs} with $g=\dd_{\{0\}^m}$ and $\Th=\R^n$, which encompasses    KKT systems of 
classical nonlinear programming problems  with equality constraints. In this case, one can see via \eqref{pdp} that    \eqref{crc} simplifies as 
$$
\nabla^2_{xx} L (\ox,\olm)w\in \rge\nabla \Phi(\ox)^*,\;\; \nabla \Phi(\ox)w=0,
$$
where `$\reg$' stands for the range of a linear mapping.  Critical and noncritical Lagrange multipliers play a major role in the local convergence 
analysis of Newtonian methods including  the SQP methods. We refer our readers to \cite[Chapter~7]{is14} for detailed discussions on this subject.

We begin our analysis of noncritical multipliers of the KKT system \eqref{vs} by revealing an interesting connection between the latter concept  
and stationary points of a second-order approximation of the composite problem \eqref{comp}.

\begin{Proposition}[noncriticality  via second-order approximation]\label{soacr}
Assume that $(\ox,\olm)$ is a solution to the variational system \eqref{vs}. Then $\olm$ is a noncritical multiplier for \eqref{vs} if and only if $w=0$ is 
the unique stationary point of the  problem 
\begin{equation}\label{axic}
\mini \la \nabla^2_{xx}L (\ox,\olm)w,w\ra+\d^2g(\Phi(\ox),\olm)(\nabla \Phi(\ox)w) \quad\mbox{subject to}\;\; w\in K_\Th(\ox,-\nabla_x L(\ox,\olm)).
\end{equation}
\end{Proposition}
\begin{proof} To prove the claimed equivalence, observe that 
\begin{eqnarray*}
\nabla \Phi(\ox)^*D(\sub g) (\Phi(\ox),\olm ) (\nabla \Phi(\ox)w )&=&\nabla \Phi(\ox)^*\sub\big(\sm  \d^2 g(\Phi(\ox) , \olm )\big) (\nabla \Phi(\ox)w)\\
&=&\sm \sub_w\big( \d^2 g(\Phi(\ox) , \olm )(\nabla \Phi(\ox)\cdot)\big) (w),
\end{eqnarray*}
where the first equality comes from \eqref{gdpl} and the second one results from \cite[Corollary~3.8]{mmsmor}.   Since $\Th$ is a polyhedral convex set, the indicator function $\dd_\Th$ is CPLQ. Employing again \eqref{gdpl} tells us that 
$$
DN_\Th(\ox, -\nabla_x L(\ox,\olm))(w)=\sub\big(\sm  \d^2 \dd_\Th(\ox,-\nabla_x L(\ox,\olm))\big) (w)=\sm \sub\big(\dd_{K_\Th(\ox,-\nabla_x L(\ox,\olm))}\big)(w),
$$
where the last equality comes from \eqref{ssp}. 
It follows from Proposition~\ref{sop}(b) that $ \d^2 g(\Phi(\ox) , \olm )$ is CPLQ, which together with    \cite[Exercise~10.22(b)]{rw} shows that the function $w\mapsto \d^2 g(\Phi(\ox) , \olm )(\nabla \Phi(\ox)w)$ is CPLQ. 
So by  \cite[Exercise~10.22(a)]{rw} and the fact that $K_\Th(\ox,-\nabla_x L(\ox,\olm))$ is a polyhedral convex set , we obtain    the subdifferential sum rule 
\begin{eqnarray*}
 &&\sub_w \Big (\d^2 g(\Phi(\ox) , \olm )(\nabla \Phi(\ox)\cdot)+\dd_{K_\Th(\ox,-\nabla_x L(\ox,\olm))} \Big)(w)\\
 &=&\sub_w\big( \d^2 g(\Phi(\ox) , \olm )(\nabla \Phi(\ox)\cdot)\big) (w)+\sub\big(\dd_{K_\Th(\ox,-\nabla_x L(\ox,\olm))}\big)(w).
\end{eqnarray*}
Combining these confirms that \eqref{crc} amounts to the inclusion 
$$
0\in  \sub_w \Big(  \la \nabla^2_{xx}L (\ox,\olm)\cdot,\cdot\ra+ \d^2 g(\Phi(\ox), \olm )(\nabla \Phi(\ox)\cdot)+\dd_{K_\Th(\ox,-\nabla_x L(\ox,\olm))} \Big)(w).
$$
This clearly justifies   the claimed equivalence for the noncriticality of the Lagrange multiplier $\olm$ and so completes the proof.
\end{proof}

We continue our second-order analysis of the noncriticality of multipliers associated with \eqref{vs} by establishing another   equivalent description  of this notion. To this end, define the set-valued mapping $G:\R^n\times \R^m\tto \R^n\times \R^m$
by
 \begin{equation}\label{GKKT}
G(x,\lm):=\begin{bmatrix}
  \nabla_x L(x,\lm)\\
-\Phi(x)
\end{bmatrix}
+\begin{bmatrix}
N_\Th(x)\\
(\sub g)^{-1}(\lm)
\end{bmatrix}.
\end{equation}
It is easy to see that $(\ox,\olm)$ is a solution to the KKT system \eqref{vs} if and only if $(0,0)\in G(\ox,\olm)$.
\begin{Proposition}[proto-differentiability of KKT mappings]\label{pro1} Assume that $(\ox,\olm)$ is a solution to the KKT system \eqref{vs}.
Then the set-valued mapping $G$ from \eqref{GKKT} is proto-differentiable at $(\ox,\olm)$ for $(0,0)\in \R^n\times \R^m$ and for any $(w,u)\in \R^n\times\R^m$ its proto-derivative is calculated by 
\begin{equation}\label{pdg}
DG\big((\ox,\olm),(0,0)\big)(w,u)=\begin{bmatrix}
 \nabla^2_{xx} L(x,\lm)w+ \nabla \Phi(\ox)^*u\\
-\nabla \Phi(\ox)w
\end{bmatrix}
+
\begin{bmatrix}
DN_\Th(\ox, - \nabla_x L(\ox,\olm))(w)\\
D(\sub g)^{-1}(\olm, \Phi(\ox))(u)
\end{bmatrix},
\end{equation}
where $D(\sub g)^{-1}(\olm, \Phi(\ox))$ stands for the proto-derivative of $(\sub g)^{-1}$ at $\olm$ for  $\Phi(\ox)$.
\end{Proposition}
\begin{proof} Let $G=G_1+G_2$, where $G_1$ and $G_2$ are defined by 
$$
G_1(x,\lm)=\begin{bmatrix}
  \nabla_x L(x,\lm)\\
-\Phi(x)
\end{bmatrix}
 \quad \mbox{and}\quad G_2(x,\lm)=\begin{bmatrix}
N_\Th(x)\\
(\sub g)^{-1}(\lm)
\end{bmatrix}.
$$
Clearly,  $G_1$ is differentiable at $(\ox,\olm)$ and 
$$
\nabla G_1(\ox,\olm)=\begin{bmatrix}
\nabla^2_{xx}L(x,\lm)& \nabla \Phi(\ox)^*\\
-\nabla \Phi(\ox)& 0
\end{bmatrix}.
$$
We are going to show that $G_2$ is proto-differentiable at  $(\ox,\olm)$ for $(- \nabla_x L(\ox,\olm), \Phi(\ox))$. To do so, we first claim that 
\begin{equation}\label{tpd}
\big((w,u),(p,q)\big)\in T_{\ss\gph G_2}\big((\ox,\olm),(- \nabla_x L(\ox,\olm), \Phi(\ox))\big)\iff 
\begin{cases} (w,p)\in T_{\ss\gph N_\Th}\big(\ox, - \nabla_x L(\ox,\olm)\big),\\
(u,q)\in T_{\ss\gph (\sub g)^{-1}}\big(\olm, \Phi(\ox)\big).
\end{cases}
\end{equation}
The implication `$\implies$' follows directly from the definition of tangent cone. To prove the opposite implication, pick the pairs $(w,p)$ and $(u,q)$ from right-hand side of \eqref{tpd}.
By the latter, we find    sequences $t_k\dn 0$ and $(u_k,q_k)\to (u,q)$ such that $(\olm, \Phi(\ox))+t_k(u_k,q_k)\in \gph  (\sub g)^{-1}$ for all $k\in \N$. According to Proposition~\ref{sop}(c),
$N_\Th$ is proto-differentiable at $\ox$ for $- \nabla_x L(\ox,\olm)$. Thus, for the aforementioned sequence $\{t_k\b$, there exists a sequence $(w_k,p_k)\to (w,p)$ such that 
$(\ox,- \nabla_x L(\ox,\olm))+t_k(w_k,p_k)\in \gph N_\Th$ for all $k\in \N$. Combining these tells us that 
$$
\big((\ox, \olm), (- \nabla_x L(\ox,\olm), \Phi(\ox))\big)+t_k\big((w_k,u_k),(p_k,q_k)\big)\in \gph G_2\quad \mbox{for all}\;k\in \N,
$$
which clearly yields   $\big((w,u),(p,q)\big)\in T_{\ss\gph G_2}\big((\ox,\olm),(- \nabla_x L(\ox,\olm),\Phi(\ox))\big)$ and hence justifies \eqref{tpd}. 
To prove the proto-differentiability of $G_2$  at  $(\ox,\olm)$ for $(-\nabla_x L(\ox,\olm), \Phi(\ox))$,  it suffices to show that all the tangent vectors $\big((w,u),(p,q)\big)$ from the left-hand  side of \eqref{tpd}
are derivable. By \eqref{tpd}, this amounts to the derivability of the tangent vectors $(w,p)$ and $(u,q)$ from the right-hand side of \eqref{tpd}. To justify this, observe from Proposition~\ref{sop}(c) that 
the normal cone mapping $N_\Th$ is proto-differentiable at $\ox$ for $- \nabla_x L(\ox,\olm)$ and that the subgradient mapping $\sub g$ is proto-differentiable at $\Phi(\ox)$ for $\olm$.
By definition, these imply that both tangent vectors $(w,p)$ and $(u,q)$ are derivable, which proves  that $\big((w,u),(p,q)\big)$ is derivable. 
Appealing now to \cite[Proposition~5.2]{r89} and then using the differentiability of $G_1$ at $(\ox,\olm)$ and the proto-differentiability of $G_2$ at $(\ox,\olm)$ for $(- \nabla_x L(\ox,\olm), \Phi(\ox))$ confirm that $G$ is 
proto-differentiable at $(\ox,\olm)$ for $(0,0)$. Finally,  
we use again \cite[Proposition~5.2]{r89} to conclude for any $(w,u)\in \R^n\times \R^m$ that  
$$
DG\big((\ox,\olm),(0,0)\big)(w,u)= \nabla G_1(\ox,\olm)(w,u)+DG_2\big((\ox,\olm),(- \nabla_x L(\ox,\olm), \Phi(\ox))\big)(w,u).
$$
This along with \eqref{tpd} justifies \eqref{pdg} and so completes the proof.
\end{proof}

The  proto-derivative formula \eqref{pdg} of $G$ allows us to provide equivalent descriptions of the criticality and noncriticality of Lagrange multipliers for the KKT system \eqref{vs} as shown below.
\begin{Corollary}[equivalent descriptions of noncriticality]\label{charc}  Assume that $(\ox,\olm)$ is a solution to the KKT system \eqref{vs}. Then the following  conditions hold:
 \begin{itemize}[noitemsep,topsep=0pt]
 \item [{\rm (a)}] the multiplier $\olm$ is noncritical for the KKT system \eqref{vs} if and only if the implication 
\begin{equation}\label{equnc}
 (0,0)\in DG\big((\ox,\olm),(0,0)\big)(w,u) \implies w=0
\end{equation}
holds, where $G$ is taken from \eqref{GKKT} and $(w,u)\in \R^n\times \R^m$;
 \item [{\rm (b)}] the multipliers $\olm$ is noncritical to the KKT system \eqref{vs} if and only if we have 
 $$
 \begin{cases}
 0\in\nabla^2_{xx} L (\ox,\olm)w+\nabla \Phi(\ox)^*u + N_{K_\Th(\ox, - \nabla_x L(\ox,\olm))}(w),\\
 u\in D(\sub g) (\Phi(\ox),\olm ) (\nabla \Phi(\ox)w )
 \end{cases}
 \implies w=0.
 $$
 \end{itemize}
\end{Corollary}
\begin{proof} Part (a) follows from \eqref{pdg} and part (b) results from \eqref{pdp}.
\end{proof}
 
Note that if, in addition,  for all $i=1,\ldots,s$ we have the matrices $A_i=0$ in \eqref{PWLQ}, the CPLQ function $g$ reduces to a convex piecewise linear function.
In this case, the characterization of the noncriticality in Corollary~\ref{charc}(b) can be considerably simplified; see \cite[Theorem~3.3]{ms17}. 
Note also that in contrast with \eqref{equnc} the stronger implication 
\begin{equation*}
 (0,0)\in DG\big((\ox,\olm),(0,0)\big)(w,u) \implies w=0, \, u=0
\end{equation*}
is equivalent  by \cite[Theorem~4E.1]{dr} to the strong metric subregularity of $G$ at $\big((\ox,\olm),(0,0)\big)$. The latter  yields  $\Lm(\ox)=\{\olm\}$, meaning that the Lagrange multipliers associated with 
$\ox$ have to be unique. Observe that the implication \eqref{equnc} does not impose such a restriction on the Lagrange multiplier set $\Lm(\ox)$.

Below, we discuss the noncriticality of another important class of problems, which fits into the composite problem \eqref{comp}. 

\begin{Example}[noncriticality in extended linear-quadratic programming]\label{elqp}{\rm Suppose that the functions $\ph$,   $\Phi$, and $g$ from \eqref{comp} are given, respectively,  by 
$$
\ph(x)=\la q,x\ra+\sm \la Qx,x\ra, \;\; \Phi(x)=b-Ax \quad \mbox{with}\; x\in \R^n, 
$$
and 
\begin{equation}\label{theta}
g(z):=f_{\ss {\O,B}}(z)=\underset{u\in \O}\sup{\big\{\la{z},{u}\ra-\sm\la{u},{Bu}\ra\big\}}\quad \mbox{with}\; z\in \R^m, 
\end{equation}
where $q\in \R^n$, $b\in \R^m$, $Q$ is an $n\times n$ symmetric matrix, $A$ is an $n\times m$ matrix,  and where   $\O$  is  a polyhedral convex set in $\R^m$ and $B$ is   an  $m\times m$ symmetric  and positive-semidefinite matrix.
The composite problem \eqref{comp} with these initial data falls into the class of extended linear-quadratic programming problems, which goes back to Rockafellar and Wets \cite{rw86}. We know from \cite[Example~11.18]{rw}
that $g$ is CPLQ. If $(\ox,\olm)$ is a solution to the KKT system \eqref{vs}, adjusted for the given functions, then it follows from \eqref{gdpl} that 
\begin{eqnarray*}
&&D(\sub g) (\Phi(\ox),\olm ) (\nabla \Phi(\ox)w )=\sub\big(\sm  \d^2 f_{\ss {\O,B}}(\Phi(\ox),\olm )\big) (\nabla \Phi(\ox)w)\\
&=&\sub f_{\ss K_\O(\Phi(\ox),\Phi(\ox)-B\olm), B} \big(\nabla \Phi(\ox)w\big)=\big(N_{K_\O(\Phi(\ox),\Phi(\ox)-B\olm)}+B\big)^{-1}(\nabla \Phi(\ox)w),
\end{eqnarray*}
where the second equality comes from  \cite[Example~13.23]{rw}  and the last one comes form  \cite[Example~11.18]{rw}.
Using this and Corollary~\ref{charc}(b) tells us that 
$\olm$ is a noncritical   multiplier for \eqref{vs}
if and only if the following implication holds:
$$
 \begin{cases}
 0\in Qw-A^*u + N_{K_\Th(\ox, - \nabla_x L(\ox,\olm))}(w),\\
0\in Aw+Bu+ N_{K_\O(\Phi(\ox),\Phi(\ox)-B\olm)}(u)
 \end{cases}
 \implies w=0.
 $$
}
\end{Example}
We proceed now with a characterization of noncriticality of Lagrange multipliers via the calmness of the solution mapping $S:\R^n\times \R^m\tto \R^n\times \R^m$,
defined by 
\begin{equation}\label{maps}
S(v,p):=\big\{(x,\lm)\in\R^n\times\R^m\;\big|\;(v,p)\in G(x,\lm)\big\}\quad\mbox{with}\;\;(v,p)\in\R^n\times\R^m,
\end{equation}
where the mapping $G$ is taken from \eqref{GKKT}. In fact, 
the mapping $S$ can be viewed as the {\em solution map} to the   KKT system  of  the canonical perturbation of the composite  optimization problem \eqref{comp}, namely the problem 
\begin{equation}\label{pcoop}
\mbox{minimize}\;\;\;\ph_0(x)+g(\Phi(x)+p)-\la v,x\ra\quad\mbox{subject to}\;\;x\in\Th.
\end{equation}

\begin{Theorem}[characterization of noncriticality via clamness]\label{chcri} Assume 
 that $(\ox,\olm)$ is a solution to the KKT system \eqref{vs}. Then the following  conditions are equivalent:
 \begin{itemize}[noitemsep,topsep=0pt]
 \item [{\rm (a)}] the multiplier $\olm$ is noncritical for the KKT system \eqref{vs};
 \item [{\rm (b)}] there are neighborhoods $U$ of $(0,0)\in \R^n\times \R^m$ and $V$ of $(\ox,\olm)$ and a constant $\kappa\ge 0$ such that 
 \begin{equation}\label{semi}
 S(v,p)\cap V\subset\big(\{\ox\}\times\Lm(\ox)\big)+\kappa\big(\|v\|+\|p\|\big)\B
 \end{equation}
 holds for all $(v,p)\in U $.
 \end{itemize}
\end{Theorem}
\begin{proof} We begin by proving the implication (b)$\implies $(a). By Proposition~\ref{charc}(a), it suffices to show that \eqref{equnc} is satisfied.
To do so, pick $ (0,0)\in DG\big((\ox,\olm),(0,0)\big)(w,u) $ with $(w,u)\in \R^n\times \R^m$. So we find sequences $t_k\dn 0$ and $\big((w_k,u_k),(v_k,p_k)\big)\to \big((w,u),(0,0)\big)$
as $k\to \infty$ such that 
$$
\big((\ox,\olm),(0,0)\big)+t_k \big((w_k,u_k),(v_k,p_k)\big)\in \gph G\quad \mbox{for all}\; k\in \N.
$$
This clearly implies via \eqref{maps} that $(\ox+t_kw_k,\olm+t_ku_k)\in S(t_kv_k,t_kp_k)$ for all $k\in \N$. Using this together with \eqref{semi} indicates that 
$$
\|\ox+t_kw_k-\ox\|\le \kappa t_k(\|v_k\|+\|p_k\|)
$$ 
for all $k$ sufficiently large. The latter inequality clearly yields $w=0$. This  proves  \eqref{equnc} and hence (b) holds. 

Turning to the opposite implication, assume that (a) is satisfied. First we claim that there are a constant  $\kappa\ge 0$ and neighborhoods $U$ of $(0,0)$
and $V$ of $(\ox,\olm)$ such that for any $(v,p)\in U $ and any $(x,\lm)\in S(v,p)\cap V$ we have the estimate
\begin{equation}\label{upper2}
\|x-\ox\|\le\kappa\big(\|v\|+\|p\|\big).
\end{equation}
Suppose by contradiction that the claimed estimate  fails. Thus  for any $k\in \N$, there 
are   $(v_{k},p_{k})\in\B_{1/k}(0,0)$ and $(x_k,\lm_k)\in S(v_{k},p_{k})\cap\B_{1/k}(\ox,\olm)$ satisfying
$$
\frac{\|x_k-\ox\|}{\|v_{k}\|+\|p_{k}\|}\to\infty\;\mbox{ as }\;k\to\infty.
$$
Set $t_k:=\|x_k-\ox\|$ and hence obtain  $v_{k}=o( t_k)$ and $p_{k}=o(t_k)$. By passing to a subsequence if necessary, we can assume with no harm  that
\begin{equation*}\label{xi}
\frac{x_k-\ox}{t_k}\to w\;\mbox{ as }\;k\to\infty\;\mbox{ for some }\;0\ne w\in\R^n.
\end{equation*}
Since $(x_k,\lm_k)\in S(v_{k},p_{k})$ and $(\ox,\olm)\in S(0,0)$, we conclude from \eqref{maps} that 
$$
v_k-\nabla_x L(x_k,\lm_k)\in N_\Th(x_k)\quad \mbox{and}\quad -\nabla_x L(\ox,\olm)\in N_\Th(\ox).
$$
It follows from these and  the reduction lemma for a 
polyhedral convex set (see \eqref{rl1} or  \cite[Lemma~2E.4]{dr}) that for all $k$ sufficiently large we have 
$$
v_k-\big(\nabla_x L(x_k,\lm_k) -\nabla_x L(\ox,\olm)\big)\in N_{K_\Th(\ox, -\nabla_x L(\ox,\olm))}(x_k-\ox).
$$
By the definition of the Lagrangian $L$ from \eqref{lag}, we obtain 
\begin{eqnarray*}
 \nabla_x L(x_k,\lm_k)-\nabla_x L(\ox,\olm)& =& \nabla_x L(x_k,\olm)-\nabla_x L(\ox,\olm)+\nabla \Phi(x_k)^*(\lm_k-\olm)\\
&=&\nabla^2_{xx} L(\ox,\olm)(x_k-\ox)+\nabla \Phi(\ox)^*(\lm_k-\olm)+o(t_k).
\end{eqnarray*}
Combining these and remembering that $v_{k}=o( t_k)$ result in 
\begin{equation}\label{inc1}
\frac{o( t_k)}{t_k}-\nabla^2_{xx} L(\ox,\olm)\big(\frac{x_k-\ox}{t_k}\big)-\nabla \Phi(\ox)^*\big(\frac{\lm_k-\olm}{t_k}\big)\in N_{K_\Th(\ox, -\nabla_x L(\ox,\olm))}\big(\frac{x_k-\ox}{t_k}\big)
\end{equation}
for all $k$ sufficiently large. This tells us that $(x_k-\ox)/t_k\in K_\Th(\ox, -\nabla_x L(\ox,\olm))$ and thus $w\in K_\Th(\ox, -\nabla_x L(\ox,\olm))$. Moreover, since
$\Th$  is a polyhedral convex set, so is the critical cone $K_\Th(\ox, -\nabla_x L(\ox,\olm))$. Thus we get  the inclusion 
$$
N_{K_\Th(\ox, - \nabla_x L(\ox,\olm))}\big(\frac{x_k-\ox}{t_k}\big)\subset N_{K_\Th(\ox, - \nabla_x L(\ox,\olm))}(w)
$$
for all $k$ sufficiently large, which in combination with \eqref{inc1} implies that 
\begin{equation}\label{inc3}
\frac{o( t_k)}{t_k}-\nabla^2_{xx} L(\ox,\olm)\big(\frac{x_k-\ox}{t_k}\big)\in N_{K_\Th(\ox, - \nabla_x L (\ox,\olm))}(w)+\nabla \Phi(\ox)^*\big(\frac{\lm_k-\olm}{t_k}\big).
\end{equation}
To deal with the second term in the right-hand side of this inclusion, we utilize again $(x_k,\lm_k)\in S(v_{k},p_{k})$ and $(\ox,\olm)\in S(0,0)$ to conclude via \eqref{maps}, respectively, that 
$$
\lm_k\in \sub g(z_k)\quad \mbox{and}\quad \olm\in \sub g(\oz)\quad \mbox{with}\;\; z_k:= \Phi(x_k)+p_k,\;\oz:= \Phi(\ox)
$$
Using the established reduction lemma for CPLQ functions in Theorem~\ref{rlcp} tells us that for all $k$ sufficiently large we have 
\begin{equation}\label{inc2}
\frac{\lm_k-\olm}{t_k}\in D(\sub g)(\oz,\olm)\big(\frac{z_k-\oz}{t_k}\big).
\end{equation}
This, in particular, indicates that ${(z_k-\oz)}{/t_k}\in \dom D(\sub g)(\oz,\olm)=K_g(\oz,\olm)$, where the last equality comes from Proposition~\ref{sop}(c).
Since $K_g(\oz,\olm)$ is a polyhedral convex set and since  $p_{k}=o(t_k)$ and $(z_k-\oz)/t_k\to \nabla \Phi(\ox)w$ as $k\to \infty$, we arrive at  $\nabla \Phi(\ox)w\in K_g(\oz,\olm)=\dom D(\sub g)(\oz,\olm)$.
Appealing now to  the outer Lipschitzian property of the proto-derivative $D(\sub g)(\oz,\olm)$, obtained in Proposition~\ref{ouli}(b),  and to the fact that $\nabla \Phi(\ox)w\in K_g(\oz,\olm)$ confirms the existence of a constant 
$\ell\ge 0$ such that for all $k$ sufficiently large the inclusion 
$$
D(\sub g)(\oz,\olm)\big(\frac{z_k-\oz}{t_k}\big)\subset D(\sub g)(\oz,\olm)(\nabla \Phi(\ox)w)+\ell\big\|\frac{z_k-\oz}{t_k}-\nabla \Phi(\ox)w\big\|\B
$$
 holds. This inclusion along with \eqref{inc2} leads us to 
 \begin{eqnarray*}
 \nabla \Phi(\ox)^*\Big(\frac{\lm_k-\olm}{t_k}\Big)&\in &  \nabla \Phi(\ox)^*D(\sub g)(\oz,\olm)\big(\frac{z_k-\oz}{t_k}\big)\\
 &\subset & \nabla \Phi(\ox)^*\Big( D(\sub g)(\oz,\olm)(\nabla \Phi(\ox)w)+\ell\big\|\frac{z_k-\oz}{t_k}-\nabla \Phi(\ox)w\big\|\B \Big)\\
   &= & \nabla \Phi(\ox)^*D(\sub g)(\oz,\olm)(\nabla \Phi(\ox)w)+\ell\big\|\frac{z_k-\oz}{t_k}-\nabla \Phi(\ox)w\big\| \nabla \Phi(\ox)^*\B.
 \end{eqnarray*}
 Thus, we conclude from these relationships and \eqref{inc3} that for any $k$ sufficiently large there is a $b_k\in \B$ such that 
\begin{eqnarray}
 \frac{o( t_k)}{t_k}-\nabla^2_{xx} L(\ox,\olm)\big(\frac{x_k-\ox}{t_k}\big) -\ell\big\|\frac{z_k-\oz}{t_k}-\nabla \Phi(\ox)w\big\| \nabla \Phi(\ox)^*b_k\nonumber\\
 \in N_{K_\Th(\ox, - \nabla_x L(\ox,\olm))}(w)+ \nabla \Phi(\ox)^* D(\sub g)(\oz,\olm)(\nabla \Phi(\ox)w)\label{inc4}.
 \end{eqnarray}
 By Proposition~\ref{sop}(c), the proto-derivative $D(\sub g)(\oz,\olm)(\nabla \Phi(\ox)w)$ is a polyhedral convex set and so is $ \nabla \Phi(\ox)^* D(\sub g)(\oz,\olm)(\nabla \Phi(\ox)w)$.
 Since the normal cone $N_{K_\Th(\ox, - \nabla_x L(\ox,\olm))}(w)$ is also a polyhedral convex set, the set   on the right-hand side of  \eqref{inc4} is a polyhedral convex set and so  is closed.
 Passing to a subsequence if necessary, we can assume without loss of generality that the sequence $\{b_k\b$ is convergent. Letting $k\to \infty$ in \eqref{inc4} tells us that 
 $$
 0\in\nabla^2_{xx} L(\ox,\olm)w + \nabla \Phi(\ox)^* D(\sub g)(\oz,\olm)(\nabla \Phi(\ox)w)+N_{K_\Th(\ox, - \nabla_x L(\ox,\olm))}(w),
 $$
 a contradiction with the noncriticality of the multiplier $\olm$ since $w\neq 0$. This proves \eqref{upper2}. To justify \eqref{semi}, pick the neighborhoods $U$ and $V$ 
 from \eqref{upper2} and let $(v,w)\in U $ and any $(x,\lm)\in S(v,p)\cap V$. This results in via \eqref{maps} that $\lm\in \sub g(\Phi(x)+p)$ and $v-\nabla_x L(x,\lm)\in N_\Th(x)$. Shrinking the neighborhoods $U$ and $V$
 if necessary, we conclude from Proposition~\ref{ouli}(a) and the polyhedrality of $\Th$, respectively, that 
\begin{equation}\label{inc5}
  \sub g(\Phi(x)+p)\subset \sub g(\Phi(\ox))+\ell\| \Phi(x)+p-\Phi(\ox)\|\B\quad \mbox{and}\quad N_\Th(x)\subset N_\Th(\ox)
\end{equation}
 for some constant $\ell\ge 0$.
 This together with $\lm\in \sub g(\Phi(x)+p)$ ensures that $\lm=\lm'+\ell\| \Phi(x)+p-\Phi(\ox)\|b$ for some $\lm'\in \sub g(\Phi(\ox))$ and $b\in \B$. 
 Furthermore, we can assume by shrinking $U$ and $V$ again that there is  a constant $\ell'\ge 0$ such that for  any $(x,\lm)\in S(v,p)\cap V$ with $(v,w)\in U $ we have 
 \begin{equation}\label{inc6}
 \begin{cases}
  \| \Phi(x)-\Phi(\ox)\| \le \ell'\|x-\ox\|, \;\; \|\nabla\ph(x)-\nabla \ph(\ox)\|\le \ell'\|x-\ox\|, \\
  \|\nabla \Phi(x)-\nabla \Phi(\ox)\|\le \ell'\|x-\ox\|,\;\; \|\lm\|\le \ell'.
  \end{cases}
\end{equation}
 Observe also that the Lagrange multiplier set $\Lm(\ox)$ from \eqref{laset} can be equivalently expressed as 
 $$
 \Lm(\ox)=\O\cap \sub g(\Phi(\ox))\quad \mbox{with}\quad \O:=\big\{\lm\in\R^m\;\big|\; 0\in \nabla_x L(\ox,\lm)+N_\Th(\ox)\big\}.
 $$
Since both   $\O$ and $\sub g(\Phi(\ox))$ are polyhedral convex sets, it follows from \cite[Theorem~8.35]{io}
that there is a constant $\rho\ge 0$ such that 
$$
\dist(\lm, \Lm(\ox))\le \rho\big(\dist(\lm,\O)+ \dist(\lm, \sub g(\Phi(\ox))\big).
$$
Using the classical Hoffman lemma (cf. \cite[Lemma~3C.4]{dr}) gives a constant $\rho'\ge 0$ such that 
$$
\dist(\lm,\O) \le \rho' \dist\big(-\nabla_xL(\ox,\lm),N_\Th(\ox)\big).
$$
Combining these and using $\lm=\lm'+\ell\| \Phi(x)+p-\Phi(\ox)\|b$, $v-\nabla_xL(x,\lm)\in N_\Th(x)\subset N_\Th(\ox)$, \eqref{inc5}, and \eqref{inc6}, we arrive at the estimates 
\begin{eqnarray}
\dist(\lm, \Lm(\ox))&\le & \rho'' \Big(\dist\big(-\nabla_xL(\ox,\lm),N_\Th(\ox) \big)+\dist(\lm, \sub g(\Phi(\ox))\Big)\label{polysb}\\
&\le & \rho'' \Big( \| \nabla_xL(\ox,\lm)-\nabla_xL(x,\lm)+v\| +\ell\| \Phi(x)+p-\Phi(\ox)\| \Big)\nonumber\\
&\le & \rho'' \max\{\ell',1\} \Big(  \|\nabla \ph(x)-\nabla \ph(\ox)\|+ \|\lm\|\| \nabla \Phi(x)-\nabla \Phi(\ox)\|\nonumber\\
&&+\|v\| +\| \Phi(x)-\Phi(\ox)\|+\|p\|\Big) \nonumber\\
&\le & \rho''\max\{\ell',1\}  \Big( (2\ell'+\ell'^2)\|x-\ox\|+ \|v\|+\|p\|\Big) \nonumber\\
&\le & \rho'' \max\{\ell',1\}  \Big( (2\ell'+\ell'^2) \kappa(\|v\|+\|p\|)+ \|v\|+\|p\|\Big) \nonumber\\
&=&  \rho'' \max\{\ell',1\}  ( (2\ell'+\ell'^2) \kappa+1)\big(\|v\|+\|p\|\big) \nonumber,
\end{eqnarray}
where $\rho'':=\rho\max\{\rho',1\}$ and where the last inequality results from \eqref{upper2}. This estimate along with \eqref{upper2} justifies the claimed calmness of the solution mapping $S$ in (a) for 
the neighborhoods $U$ and $V$ and thus completes the proof.
\end{proof}

\begin{Remark}[characterization of noncriticality of variational systems]{\rm It is valuable to mention  that the given proof for Theorem~\ref{chcri} can be used to achieve a similar characterization 
of noncritical multipliers of the variational system 
\begin{equation}\label{vs2} 
0\in \Psi(x,\lm)+N_\Th(x),\;\;\lm\in \sub g\big(\Phi(x)\big)
\quad \mbox{with}\;\;\Psi(x,\lm):=   f(x)+\nabla \Phi(x)^*\lm,
\end{equation}
where $f:\R^n\to \R^n$ is a differentiable function and where   $g$, $\Phi $, and $\Th$ are taken from \eqref{comp}.
The critical and noncritical Lagrange  multipliers for \eqref{vs2} can be  defined as of those for  the KKT system \eqref{vs}.
While  reducing to the KKT system \eqref{vs} for  $f=\nabla \ph$,  the variational system \eqref{vs2} has important applications in sensitivity analysis of 
variational inequalities. 
}
\end{Remark}

When $\Th=\R^n$ and the CPLQ function $g$ is defined by \eqref{theta}, 
 the established characterization of the noncriticality in Theorem~\ref{chcri} boils down to \cite[Theorem~5.1]{dms}, where the idea of using the reduction lemma for a polyhedral convex set
in the characterization of the noncriticality was first appeared. Using similar approach for the composite problem \eqref{comp}
requires a counterpart of the reduction lemma for CPLQ functions, which was achieved in Theorem~\ref{rlcp}. 
When $g$ enjoys this representation, \eqref{vs} can cover the KKT systems of an important class of composite optimization problems, 
called  {\em extended nonlinear programs}; see \cite{r97} for more details and discussion about this class of optimization problems. 
  When $\Th=\R^n$ and $A_i=0$ for all $i=1,\ldots,s$ in \eqref{PWLQ}, meaning that   $g$ is piecewise linear, 
Theorem~\ref{chcri} reduces  to \cite[Theorem~4.1]{ms17}. The choices of $g=\dd_{\{0\}^s\times \R_-^{m-s}}$ and $\Th=\R^n$ for some $0\le s\le m$ allow to reduce the composite 
problem \eqref{comp} into a nonlinear programming problem with the $s$ equality constraints and the $m-s$ inequality constraints 
for which similar characterization of the noncriticality can be found in \cite[Theorem~1.40]{is14}. 

It is well known that the calmness of a set-valued mapping is equivalent to the metric subregularity of its inverse mapping (cf. \cite[Theorem~3H.3]{dr}).
This motivates us  to look for  an equivalent error bound estimate  of the calmness property \eqref{semi} of  the solution mapping $S$.
To do so, recall that the proximal mapping of  a function $f:\R^n\to \oR$ is defined by 
\begin{equation*}\label{pr}
\prox_f(x):={\rm argmin}_{z\in \R^n}\big\{f(z)+\sm\|x-z\|^2\big\},\quad x\in\R^n.
\end{equation*}

 \begin{Proposition}[error bound for KKT systems]\label{error}
Assume  that $(\ox,\olm)$ is a solution to the KKT system \eqref{vs}. Then the following  conditions are equivalent: 
 \begin{itemize}[noitemsep,topsep=0pt]
 \item [{\rm a)}] there are neighborhoods $U$ of $(0,0)\in \R^n\times \R^m$ and $V$ of $(\ox,\olm)$ and a constant $\kappa\ge 0$ such that 
 the solution mapping $S$ from \eqref{maps} satisfies the calmness property \eqref{semi};
\item[\rm b)]  there are numbers $\ve>0$ and $\kappa\ge 0$ such that the error bound estimate
\begin{equation}\label{subr3}
\|x-\ox\|+\dist\big(\lm,\Lambda(\ox)\big)\le\kappa\big(\dist\big(-\nabla_x L(x,\lm),N_\Th(x)\big)+\| \Phi(x)-\prox_g(\lm+ \Phi(x))\|\big)
\end{equation}
holds for any $(x,\lm)\in\B_\ve(\ox,\olm)$.
 \end{itemize}
\end{Proposition}
\begin{proof}  Assume first that (b) holds. Pick $(v,p)\in \B_\ve(0,0)$ and $(x,\lm)\in \B_\ve(\ox,\olm)\cap S(v,p)$ with   $\ve$ taken from (b)
and conclude  via \eqref{maps} that $\lm\in \sub g(\Phi(x)+p)$ and $v-\nabla_x L(x,\lm)\in N_\Th(x)$. The former together with  $\prox_g=(I+\sub g)^{-1}$ (cf. \cite[Proposition~12.19]{rw})
yields  $\prox_g(\lm+ \Phi(x)+p)= \Phi(x)+p$. Appealing now to \eqref{subr3} and shrinking $\ve$ if necessary  to secure the inclusion $N_\Th(x)\subset N_\Th(\ox)$ bring us to the estimates
\begin{eqnarray*}
\|x-\ox\|+\dist\big(\lm,\Lm(\ox)\big)&\le&\kappa\big(\dist\big(-\nabla_x L(x,\lm),N_\Th(x)\big)+\| \Phi(x)-\prox_g(\lm + \Phi(x)\|\big)\\
&\le&\kappa\big(\|-\nabla_x L(x,\lm)-v+\nabla_x L(x,\lm)\|\\
&&+\|\prox_g(\lm+ \Phi(x)+p)-\prox_g(\lm + \Phi(x))\| +\|p\|\big)\\
&\le&2\kappa\big(\|v\|+\|p\|\big),
\end{eqnarray*}
and thus prove (a). 

Suppose now that (a) is satisfied. 
Since $\nabla \Phi $ is continuous at $\ox$, we find some constants $\ve>0$ and  $\rho>1$ for which  we have $\|\nabla \Phi(x)\|\leq \rho$ for all $x\in\B_{\ve}(\ox)$.
Shrinking $\ve$ if necessary, we assume without loss of generality that 
 $ \B_{\ve/\sqrt{(\rho+1)^2+1}}(0,0)\subset U$ and $\B_{\ve/\rho}(\ox,\olm)\subset V$, where $U$ and $V$ come from  (a). Pick $(x,\lm)\in\B_{\ve/8\rho}(\ox,\olm)$ and set $p:= \prox_g\big(\lm+ \Phi(x)\big)-\Phi(x)$.
  Since $\Phi $ and  $\prox_g$  are continuous and since  $\prox_g\big(\olm+ \Phi(\ox)\big)= \Phi(\ox)$, we can assume by shrinking $\ve$ if necessary that $p\in \B_{\ve/2\rho}(0)$.
 Moreover, the definition of $p$ and the identity  $\prox_g=(I+\sub g)^{-1}$  tell us that $\lm-p\in \sub g(\Phi(x)+p)$.
Suppose that $\big(\nabla_x L(x,\lm)+N_\Th(x)\big)\cap \B_{\ve/2}(0)\neq \emptyset$ and so  choose   $u\in \big(\nabla_x L(x,\lm)+N_\Th(x)\big)\cap \B_{\ve/2}(0)$ such that 
\begin{equation*}\label{ee1}
\dist\Big(0,\big(\nabla_x L(x,\lm)+N_\Th(x)\big)\cap \B_{\ve/2}(0)\Big)=\|u\|.
\end{equation*}
Thus we have  $(v,p)\in \B_{\ve/\sqrt{(\rho+1)^2+1}}(0,0)\subset U $ with $v:=u-\nabla \Phi(x)^*p $. By the definitions of $p$ and $v$, it follows from \eqref{maps} that   $(x,\lm-p)\in S(v,p)\cap V$.
Using the Lipschitz continuity of the distance function together with (a) yields the estimates
\begin{eqnarray*}
\|x-\ox\|+\dist\big(\lm,\Lm(\ox)\big)&\le&\|x-\ox\|+\dist\big(\lm-p,\Lm(\ox)\big) +\|p\|\\
 &\le&\kappa\big(\|v\|+\|p\|\big)+\|p\|\\
&\le&(\kappa\rho+\kappa+1)\Big(\dist\big(0,\big(\nabla_x L(x,\lm)+N_\Th(x)\big)\cap \B_{\ve/2}(0)\big)\\
&&+\| \Phi(x)-\prox_g(\lm + \Phi(x))\|\Big)\\
&=&(\kappa\rho+\kappa+1) \Big(\dist\big(-\nabla_x L(x,\lm),N_\Th(x)\big)+ \| \Phi(x)-\prox_g(\lm + \Phi(x))\|\Big),
\end{eqnarray*}
where the last equality comes from   $\big(\nabla_x L(x,\lm)+N_\Th(x)\big)\cap \B_{\ve/2}(0)\neq \emptyset$, which implies that 
$$
\dist\big(0,\big(\nabla_x L(x,\lm)+N_\Th(x)\big)\cap \B_{\ve/2}(0)\big)=\dist\big(0,\nabla_x L(x,\lm)+N_\Th(x)\big)=\dist\big(-\nabla_x L(x,\lm),N_\Th(x)\big).
$$
The above estimates prove (b) for all $(x,\lm)\in\B_{\ve/8\rho}(\ox,\olm)$ with $\big(\nabla_x L(x,\lm)+N_\Th(x)\big)\cap \B_{\ve/2}(0)\neq \emptyset$.
If the latter condition fails,   we conclude for all  $(x,\lm)\in\B_{\ve/8\rho}(\ox,\olm)$ that 
\begin{eqnarray*}
\dist\big(-\nabla_x L(x,\lm),N_\Th(x)\big)> \frac{\ve}{2}\ge \frac{\ve}{8\rho}\ge  \|x-\ox\|+\|\lm-\olm\|\ge \|x-\ox\|+ \dist\big(\lm,\Lm(\ox)\big).
\end{eqnarray*}
This clearly verifies \eqref{subr3} for this case and hence completes the proof.
\end{proof}  

Our next goal is to explore the relationship between the noncriticality of a Lagrange multiplier and the second-order sufficient condition for the composite problem \eqref{comp}.
The latter, as shown in the coming sections, plays a major role in the convergence analysis of the basic SQP method for this problem. Given  a solution $(\ox,\olm)$  to \eqref{vs}, the second-order sufficient 
for the composite problem \eqref{comp} at $(\ox,\olm)$ is formulated by 
\begin{equation}\label{sosc}
\langle\nabla_{xx}^2L(\bar x,\olm)w,w\rangle+\d^2g (\Phi(\bar x),\olm) (\nabla \Phi(\ox)w )>0\quad \mbox{for all }\; w\in \D\setminus\{0\},
\end{equation}
where the convex cone $\D$ comes from \eqref{coned}.
We show below that the second-order sufficient condition \eqref{sosc} yields the noncriticality of Lagrange multipliers.
\begin{Proposition}[noncriticality via second-order sufficient conditions]\label{nocs} Assume that $(\ox,\olm)$ is a solution to the KKT system \eqref{vs}.
If the second-order sufficient condition \eqref{sosc} holds at $(\ox,\olm)$, then $\olm$ is a noncritical Lagrange multiplier for \eqref{vs}.
\end{Proposition}
\begin{proof} To justify this, pick a $w\in \R^n$ satisfying \eqref{crc}. We are going to show that $w=0$. To this end, by \eqref{crc}, we find 
$u\in D(\sub g) (\Phi(\bar x),\olm) (\nabla \Phi(\ox)w )$ and $q\in  DN_\Th(\ox,-\nabla_xL(\ox,\olm)) (w )$ for which we have 
\begin{equation}\label{soeq}
\langle\nabla^2_{xx}L(\ox,\olm)w,w\rangle+\la u,\nabla \Phi(\ox)w\ra+\la q,w\ra=0.
\end{equation}
We claim now 
\begin{equation}\label{ssr}
\begin{cases}
\la u,\nabla \Phi(\ox)w\ra=\d^2g (\Phi(\bar x),\olm) (\nabla \Phi(\ox)w )\quad \mbox{and}\\
 \la q,w\ra= \d^2 \dd_\Th(\ox,-\nabla_x L(\ox,\olm))(w)=\dd_{K_\Th(\ox,-\nabla_x L(\ox,\olm))}(w).
 \end{cases}
\end{equation}
Indeed, the last equality also  falls directly out of \eqref{ssp}. 
 To prove the second equality, observe  from \eqref{pdp}  that $q\in N_{K_\Th(\ox,-\nabla_xL(\ox,\olm))}(w )$. Since the critical cone $K_\Th(\ox,-\nabla_xL(\ox,\olm))$ is convex, we get $\la q,w\ra=0$,
 which together with $w\in K_\Th(\ox,-\nabla_xL(\ox,\olm))$ justifies the second equality.  
To justify  the first equality in \eqref{ssr}, we conclude from    \eqref{gdpl}  that $u\in \sub\big(\sm  \d^2g (\Phi(\bar x),\olm) \big)(\nabla \Phi(\ox)w )$. This along with Proposition~\ref{sop}(b) tells us 
that 
\begin{equation}\label{pe80}
\nabla \Phi(\ox)w\in\dom \d^2g (\Phi(\bar x),\olm)  =K_g(\Phi(\bar x),\olm).
\end{equation}
Using again Proposition~\ref{sop}(b)  shows that   $\d^2g (\Phi(\bar x),\olm)$ is a convex function. By  the definition of the subdifferential in convex analysis, we arrive at 
\begin{equation*}
\la u,v-\nabla \Phi(\ox)w\ra\le\sm \d^2g (\Phi(\bar x),\olm) (v)-\sm \d^2g (\Phi(\bar x),\olm) (\nabla \Phi(\ox)w)\quad \mbox{for all}\;\; v\in \R^m.
\end{equation*}
Let $\ve\in(0,1)$ and set $v:=(1\pm\ve)\nabla \Phi(\ox)w$. Since the second subderivative is positive homogeneous of degree $2$, the above inequality leads us to 
\begin{equation*}
\pm\la u,\nabla \Phi(\ox)w\ra\le\frac{\ve\pm 2}{2} \d^2g (\Phi(\bar x),\olm) (\nabla \Phi(\ox)w ),
\end{equation*}
which in turn results in  $\d^2 g(\Phi(\ox),\bar\lm)(\nabla \Phi(\ox)w)=\la u,\nabla \Phi(\ox)w\ra$ by letting $\ve\dn 0$. This   proves the first equality in \eqref{ssr}.
Combining \eqref{soeq}-\eqref{pe80} brings us to 
$$
\langle\nabla^2_{xx}L(\ox,\olm)w,w\rangle+\d^2g (\Phi(\bar x),\olm) (\nabla \Phi(\ox)w )=0, \quad w\in \D.
$$
By \eqref{sosc}, we conclude that $w=0$, implying that $\olm$ is a noncritical Lagrange multiplier.
\end{proof}

Note that in general the second-order sufficient condition \eqref{sosc} is strictly stronger than the noncriticality; see \cite[Example~3]{is12}
for   an example of a nonlinear program that shows this fact. These condition  are, however, equivalent when the Lagrange multiplier set $\Lambda(\ox)$ from \eqref{laset} is a singleton
and the stationary point $\ox$ is in fact a local minimum of the composite optimization problem \eqref{comp} as shown  below. 
To achieve  this goal, we are going first to present a simple but useful characterization of uniqueness of Lagrange multipliers for \eqref{comp},
which is a direct consequence  of our recent result in \cite[Theorem~8.1]{mmsmor} for \eqref{comp} with $\Th=\R^n$.  

\begin{Proposition}[characterization of uniqueness of Lagrange multipliers]\label{unic}
Assume that $(\ox,\olm)$ is a solution to the KKT system \eqref{vs}. Then the following conditions are equivalent:
 \begin{itemize}[noitemsep,topsep=0pt]
 \item [{\rm a)}] for the Lagrange multiplier set $\Lm(\ox)$ from \eqref{laset}, we have $\Lm(\ox)=\{\olm\}$;
\item[\rm b)]   the dual condition 
\begin{equation}\label{dual}
-\nabla \Phi(\ox)^*u\in K_\Th \big(\ox,-\nabla_x L(\ox,\olm) \big)^*, \;\; u\in K_ g \big(\Phi(\bar x),\olm \big)^*\implies u=0
\end{equation}
is satisfied.
 \end{itemize}
\end{Proposition}
\begin{proof}  We showed in Remark~\ref{remeq} that the composite problem \eqref{comp} can be equivalently reformulated as \eqref{comp3}. 
It is not hard to see that $\Lm(\ox)=\{\olm\}$ if and only if the set of Lagrange multipliers associated with $\ox$ for \eqref{comp3} is $\{(\omu,\olm)\}$ with $\omu:=-\nabla_x L(\ox,\olm)$.
By \cite[Theorem~8.1]{mmsmor}, the latter amounts to the dual condition 
$$
D(\sub \psi)\big((\ox, \Phi(\ox)),(\omu,\olm)\big)(0,0)\cap \ker \begin{bmatrix}I& \nabla \Phi(\ox)^*\end{bmatrix}=\big\{(0,0)\big\},
$$
where $\psi$ comes from \eqref{comp3}. Since we have $\sub \psi(\ox, \Phi(\ox))=N_\Th(\ox)\times \sub g(\Phi(\ox))$ (cf. \cite[Proposition~10.5]{rw}),  a similar argument as the proof of \eqref{tpd} shows that 
\begin{eqnarray*}
D(\sub \psi)\big((\ox, \Phi(\ox)),(\omu,\olm)\big)(0,0)&=&DN_\Th(\ox,\omu)(0)\times D(\sub g)(\Phi(\ox),\olm)(0)\\
&=&K_\Th \big(\ox,-\nabla_x L(\ox,\olm) \big)^*\times  K_ g \big(\Phi(\bar x),\olm \big)^*,
\end{eqnarray*}
where the last equality results from Proposition~\ref{sop}(c). Combining these proves the claimed equivalence. 
\end{proof}

The  dual condition \eqref{dual} was first 
introduced in \cite{ms19} for constrained optimization problems and was observed therein that it is equivalent to the {\em strict Robinson constraint qualification} (cf. see \cite[equation~(4.4)]{ms19})
for ${\cal C}^2$-cone reducible constrained optimization problems. The latter condition boils down to the  {  strict  Mangasarian--Fromovitz constraint qualification}
for classical nonlinear programming problems; see \cite[Remark~4.49]{bs} for  more detail on this subject. 

We are now in a position to present the promised equivalence between the noncriticality and the second-order sufficient condition \eqref{sosc} when the set of Lagrange multipliers of \eqref{comp}
is a singleton. 
\begin{Theorem}[equivalence between noncriticality and  second-order sufficient condition]\label{smrso}
Assume that $(\ox,\olm)$ is a solution to the KKT system \eqref{vs}. Then the following conditions are equivalent:
\begin{itemize}[noitemsep,topsep=0pt]
 \item [{ \rm (a)}] the second-order sufficient condition \eqref{sosc} holds at $(\ox,\olm)$ and  $\Lm(\ox)=\{\olm\}$;
\item[ \rm (b)]   the second-order sufficient condition \eqref{sosc} holds at $(\ox,\olm)$ and the dual condition \eqref{dual} is satisfied;
\item[ \rm (c)] the multiplier $\olm$ is noncritical for \eqref{vs}, $\Lm(\ox)=\{\olm\}$, and $\ox$ is a local minimizer of \eqref{comp};
\item[\rm  (d)] the solution mapping $S$ from \eqref{maps} is isolated calm at $\big((0,0),(\ox,\olm)\big)$ and $\ox$ is a local minimizer of \eqref{comp}.
 \end{itemize}
\end{Theorem}
\begin{proof} The equivalence between (a) and (b) results directly  from Proposition~\ref{unic}. The equivalence between (c) and (d) 
comes from Theorem~\ref{chcri} and the fact that the Lagrange multiplier set  $\Lm(\ox)$ is convex. 

Turning now to the equivalence between (a) and (c), assume first that (a) holds. Appealing to Proposition~\ref{nocs} indicates that 
$\olm$ is a noncritical multiplier for \eqref{vs}. Moreover, since $\Lm(\ox)=\{\olm\}$, Proposition~\ref{sooc}(b) tells us that $\ox$ is a local minimizer of \eqref{comp} and so we arrive at (c).

Finally, suppose that (c) is satisfied. It follows from $\Lm(\ox)=\{\olm\}$ and Proposition~\ref{unic} that the dual condition \eqref{dual} fulfills. 
Observe also that 
$$
K_g (\Phi(\ox),\olm)=\big\{u\in\R^m\big|\;\d g (\Phi(\ox) )(u)=\langle u,\bar\lm\rangle\big\}\subset\dom\d g(\Phi(\ox) )=T_{\ss\dom g} (\Phi(\ox)).
$$
This together with the definition of the critical cone $K_\Th(\ox,-\nabla_x L(\ox,\olm))$ brings us to the inclusions 
\begin{equation}\label{pe14}
N_{\ss\dom g} (\Phi(\ox))\subset K_g (\Phi(\ox),\olm)^*\quad \mbox{and}\quad N_\Th(\ox)\subset N_\Th(\ox)+ [\nabla_x L(\ox,\olm)]= K_\Th(\ox,-\nabla_x L(\ox,\olm))^*.
\end{equation}
These inclusions,  combined with the dual condition \eqref{dual}, imply the validity of the constraint qualification \eqref{rcq}. 
Remembering that $\Lm(\ox)=\{\olm\}$ and that $\ox$ is a local minimum of \eqref{comp} and appealing to Propsoition~\ref{sooc}(a) yield 
$$
\langle\nabla_{xx}^2L(\bar x,\lm)w,w\rangle+\d^2 g(\Phi(\bar x),\lm ) (\nabla \Phi(\ox)w )\ge 0\quad \mbox{for all}\; w\in \D,
$$
which in turn gives us the second-order sufficient condition \eqref{sosc} because $\olm$ is noncritical. In fact, if \eqref{sosc} fails, 
by the inequality above we find  $w\in \D\setminus\{0\}$ that is a minimizer of problem \eqref{axic}. Thus, $w$ is 
a stationary point of problem \eqref{axic}, which is not possible by Proposition~\ref{soacr} since $\olm$ is a noncritical multiplier for \eqref{vs}. This proves (a) and hence ends the proof.
\end{proof}

Note that the characterization of the  isolated calmness of the solution mapping $S$ via the second-order sufficient condition -- the equivalence between (a) and (d) in Theorem~\ref{smrso} -- was first accomplished 
 in \cite[Theorem~2.6]{dr97} for NLPs and was extended in \cite[Theorem~7.5]{ms17} for the composite problem \eqref{comp} with $g$ piecewise linear and $\Th=\R^n$.
 The latter equivalence was recently established using a different approach for \eqref{comp} with $\Th=\R^n$ in \cite[Theorem~5.1]{be}. 
 Note that instead of the isolated calmness of the solution mapping $S$
in Theorem~\ref{smrso}(d), the authors in \cite{be} used the strong metric subregularity of the mapping $G$, taken from  \eqref{GKKT}. However, since we have $S=G^{-1}$, these notions are equivalent.
It is important to notice that  the conditions (b) and (c) in Theorem~\ref{smrso} did not appear in \cite{be}. 

\section{Primal Estimates for KKT Systems}\sce  \label{sect04} 

In this section, we aim to establish sharper estimates for the solution mapping $S$ from \eqref{maps} that play  major roles in the characterization of primal superlinear convergence (see Theorem~\ref{psqp})
of the quasi-Newton SQP method via the Dennis-Mor\'e condition \eqref{dmc}  for the composite optimization problem \eqref{comp}. To achieve such a characterization via \eqref{dmc}, we need 
a calmness property of the solution mapping $S$ similar to \eqref{semi} in which $\|v\|$ is replaced with $\|P_\D(v)\|$, where $P_\D$ stands for the projection mapping onto the convex cone $\D$.
The price for achieving such a sharper estimate   involving   $P_\D$ is  that we require to assume the second-order sufficient condition \eqref{sosc}, which is strictly stronger than the noncriticality assumption in Theorem~\ref{chcri},
and that we only can obtain such an  estimate for the primal part of any pair $(x,\lm)$. The latter, however, suffices for the primal superlinear convergent of the quasi-Newton SQP method as shown in Theorem~\ref{psqp}.
Note that as  Theorem~\ref{chcri}, the  proof of the following result mainly revolves around the reduction lemma from Theorem~\ref{rlcp}.

\begin{Theorem}[primal estimates via second-order sufficient conditions]\label{prim1} Assume that $(\ox,\olm)$ is a solution to the KKT system \eqref{vs} and that 
the second-order sufficient condition \eqref{sosc} holds at $(\ox,\olm)$. Then there are neighborhoods $U$ of $(0,0)\in \R^n\times \R^m$ and $V$ of $(\ox,\olm)$ and a constant $\kappa\ge 0$ such that 
for any $(v,p)\in U$ and any  $(x,\lm)\in  S(v,p)\cap V$ the estimate 
 \begin{equation}\label{semi2}
\|x-\ox\| \le \kappa\big(\|P_\D(v)\|+\|p\|\big)
 \end{equation}
 holds, where the solution mapping $S$ comes from \eqref{maps} and where $P_\D$ stands for the projection mapping onto the convex cone $\D$, defined in  \eqref{coned}.
\end{Theorem}
\begin{proof}
Suppose by contradiction that the claimed estimate  fails. Thus  for any $k\in \N$, there 
are   $(v_{k},p_{k})\in\B_{1/k}(0,0)$ and $(x_k,\lm_k)\in S(v_{k},p_{k})\cap\B_{1/k}(\ox,\olm)$ satisfying
$$
\frac{\|x_k-\ox\|}{\|P_\D(v_{k})\|+\|p_{k}\|}\to\infty\;\mbox{ as }\;k\to\infty.
$$
Set $t_k:=\|x_k-\ox\|$ and hence obtain  $P_\D(v_{k})=o( t_k)$ and $p_{k}=o(t_k)$. By passing to a subsequence if necessary, we can assume that 
\begin{equation*}\label{xi}
\frac{x_k-\ox}{t_k}\to w\;\mbox{ as }\;k\to\infty\;\mbox{ with some }\;0\ne w\in\R^n.
\end{equation*}
Since $(x_k,\lm_k)\in S(v_{k},p_{k})$ and $(\ox,\olm)\in S(0,0)$, we conclude from \eqref{maps} that 
$$
v_k-\nabla_x L(x_k,\lm_k)=q_k\;\;\mbox{with}\;\;q_k\in N_\Th(x_k)\quad \mbox{and}\quad -\nabla_x L(\ox,\olm)\in N_\Th(\ox).
$$
It follows from these and  the reduction lemma for a 
polyhedral convex set (see \eqref{rl1} or  \cite[Lemma~2E.4]{dr}) that for all $k$ sufficiently large we have 
$$
v_k-\big(\nabla_x L(x_k,\lm_k) -\nabla_x L(\ox,\olm)\big)\in N_{K_\Th(\ox, -\nabla_x L(\ox,\olm))}(x_k-\ox).
$$
This tells us that $(x_k-\ox)/t_k\in K_\Th(\ox, -\nabla_x L(\ox,\olm))$ and thus $w\in K_\Th(\ox, -\nabla_x L(\ox,\olm))$. Since   
 $\Th$  is a polyhedral convex set, so is the critical cone $K_\Th(\ox, -\nabla_x L(\ox,\olm))$. Thus we get  the inclusion 
$$
N_{K_\Th(\ox, -\nabla_x L(\ox,\olm))}\big(\frac{x_k-\ox}{t_k}\big)\subset N_{K_\Th(\ox, -\nabla_x L(\ox,\olm))}(w),
$$
which in turn results in 
\begin{equation}\label{pe1}
q_k+\nabla_x L(\ox,\olm)=v_k-\big(\nabla_x L(x_k,\lm_k) -\nabla_x L(\ox,\olm)\big)\subset N_{K_\Th(\ox, -\nabla_x L(\ox,\olm))}(w)
\end{equation}
for all $k$ sufficiently large.  It follows from the relationships  $P_{\D}=(I+N_\D)^{-1}$ and    $P_\D(v_{k})=o( t_k)$ that $v_k+o(t_k)\in N_\D\big(o(t_k)\big)$.
Since $\D$ is a convex cone, the latter   yields $v_k+o(t_k)\in \D^*$. So by  the definition of the Lagrangian $L$, we obtain 
\begin{eqnarray}
\D^*\ni v_k+o(t_k)&=& \nabla_x L(x_k,\lm_k)-\nabla_x L(\ox,\olm)+\nabla_x L(\ox,\olm)+ q_k+o(t_k)\nonumber\\
& =& \nabla_x L(x_k,\olm)-\nabla_x L(\ox,\olm)+\nabla \Phi(x_k)^*(\lm_k-\olm)+\nabla_x L(\ox,\olm)+ q_k+o(t_k) \nonumber\\
&=&\nabla^2_{xx} L(\ox,\olm)(x_k-\ox)+\nabla \Phi(\ox)^*(\lm_k-\olm)+\nabla_x L(\ox,\olm)+ q_k+o(t_k)\label{pe3}.
\end{eqnarray}
To deal with the second term in \eqref{pe3},  we utilize again $(x_k,\lm_k)\in S(v_{k},p_{k})$ and $(\ox,\olm)\in S(0,0)$ to conclude via \eqref{maps}, respectively, that 
$$
\lm_k\in \sub g(z_k)\quad \mbox{and}\quad \olm\in \sub g(\oz)\quad \mbox{with}\;\; z_k:= \Phi(x_k)+p_k,\;\oz:= \Phi(\ox).
$$
Using the established reduction lemma for CPLQ functions in Theorem~\ref{rlcp} tells us that for all $k$ sufficiently large we have 
\begin{equation}\label{pe5}
\frac{\lm_k-\olm}{t_k}\in D(\sub g)(\oz,\olm)\big(\frac{z_k-\oz}{t_k}\big).
\end{equation}
This, in particular, indicates that ${(z_k-\oz)}{/t_k}\in \dom D(\sub g)(\oz,\olm)=K_g(\oz,\olm)$, where the last equality comes from Proposition~\ref{sop}(c).
Since $K_g(\oz,\olm)$ is a polyhedral convex set and since  $p_{k}=o(t_k)$ and $(z_k-\oz)/t_k\to \nabla \Phi(\ox)w$ as $k\to \infty$, we arrive at  $\nabla \Phi(\ox)w\in K_g(\oz,\olm)=\dom D(\sub g)(\oz,\olm)$.
This together with $w\in K_\Th(\ox, -\nabla_x L(\ox,\olm))$ tells us that $w\in \D$. By this and  \eqref{pe3}, we get 
\begin{equation}\label{pe4}
\la \nabla^2_{xx} L(\ox,\olm)(x_k-\ox),w\ra +\la \lm_k-\olm, \nabla \Phi(\ox)w\ra +\la \nabla_x L(\ox,\olm)+ q_k,w\ra +o(t_k) \le 0
\end{equation}
for all $k$ sufficiently large. Appealing also to  the outer Lipschitzian property of the proto-derivative $D(\sub g)(\oz,\olm)$, obtained in Proposition~\ref{ouli}(b),  and to the fact that $\nabla \Phi(\ox)w\in K_g(\oz,\olm)$ confirms the existence of a constant 
$\ell\ge 0$ such that for all $k$ sufficiently large the inclusion 
\begin{equation}\label{pe6}
D(\sub g)(\oz,\olm)\big(\frac{z_k-\oz}{t_k}\big)\subset D(\sub g)(\oz,\olm)(\nabla  \Phi(\ox)w)+\ell\big\|\frac{z_k-\oz}{t_k}-\nabla \Phi(\ox)w\big\|\B
\end{equation}
 holds. This inclusion and \eqref{pe5} tell us that there are $u_k\in D(\sub g)(\oz,\olm)(\nabla \Phi(\ox)w)$ and  $b_k\in \B$ such that 
 $$
 \frac{\lm_k-\olm}{t_k}=u_k+ \ell\big\|\frac{z_k-\oz}{t_k}-\nabla \Phi(\ox)w\big\| b_k.
 $$
Similar to \eqref{ssr}, we can conclude from $u_k\in D(\sub g)(\oz,\olm)(\nabla \Phi(\ox)w)$ that 
$$
\la u_k,\nabla \Phi(\ox)w\ra= \d^2g(\oz,\olm)(\nabla \Phi(\ox)w).
$$
Moreover, by \eqref{pe1}, we have $\la \nabla_x L(\ox,\olm)+ q_k,w\ra=0$ since $K_\Th(\ox, -\nabla_x L(\ox,\olm))$ is a convex cone.
Dividing both sides of \eqref{pe4} by $t_k$ and using these facts bring us to 
$$
\la \nabla^2_{xx} L(\ox,\olm)(\frac{x_k-\ox}{t_k}),w\ra + \d^2g(\oz,\olm)(\nabla \Phi(\ox)w)+  \ell\big\|\frac{z_k-\oz}{t_k}-\nabla \Phi(\ox)w\big\| \la b_k, \nabla \Phi(\ox)w\ra   +\frac{o(t_k)}{t_k} \le 0
$$ 
for all $k$ sufficiently large. Since the sequence $\{b_k\b$ is bounded, passing to a subsequence of $\{b_k\b$  if necessary and then letting $k\to \infty$ imply that 
$$
\la \nabla^2_{xx} L(\ox,\olm)w,,w\ra + \d^2g(\oz,\olm)(\nabla \Phi(\ox)w)\le 0,
$$
where $w\in \D\setminus\{0\}$. This  clearly contradicts  the second-order sufficient condition \eqref{sosc} and hence completes the proof. 
\end{proof} 
Note that Theorem~\ref{prim1} extends a similar result  in \cite[Theorem~2.3]{fis12}, which was established for NLPs, for the composite problem \eqref{comp}.
It is worth mentioning that the proof of the latter result did not appeal to the reduction lemma and utilizes the particular geometry  of constraints in NLPs. 
We now look into the possibility whether the second-order sufficient condition \eqref{sosc} can be replaced with  the noncriticality, which is strictly weaker than \eqref{sosc}.
It was observed in \cite[Example~2.1]{fis12} that the latter can not be achieved even for nonlinear programming problems, which can be covered by the composite problem \eqref{comp}.
It is, however, observed in \cite[Theorem~2.2]{fis12} that  for nonlinear programs such a replacement can be accomplished if   the convex cone $\D$ from \eqref{coned} is enlarged.
Below we show that this is achievable for the composite problem \eqref{comp} if we replace the convex cone $\D$ by the the linear subspace 
\begin{equation}\label{cones}
 \D_+:=\Big\{w\in K_\Th\big(\ox,-\nabla_x L(\ox,\olm) \big)-K_\Th \big(\ox,-\nabla_x L(\ox,\olm) \big)\Big|\; \nabla \Phi(\ox)w\in K_g(\Phi(\ox),\olm)-K_g(\Phi(\ox),\olm)\Big\},
 \end{equation}
which clearly contains $\D$. 
\begin{Theorem}[primal estimates via noncriticality]\label{prim2} Assume that $(\ox,\olm)$ is a solution to the KKT system \eqref{vs} and that 
$\olm$ is a critical Lagrange multiplier for \eqref{vs}. Then there are neighborhoods $U$ of $(0,0)\in \R^n\times \R^m$ and $V$ of $(\ox,\olm)$ and a constant $\kappa\ge 0$ such that 
for any $(v,p)\in U$ and any  $(x,\lm)\in  S(v,p)\cap V$ the estimate 
 \begin{equation}\label{semi3}
\|x-\ox\| \le \kappa\big(\|P_{\D_+}(v)\|+\|p\|\big)
 \end{equation}
 holds.
 \end{Theorem}
 \begin{proof} We can proceed as the proof of Theorem~\ref{prim1} with some small adjustments to get \eqref{pe1}.
 Indeed, assume by contradiction that the claimed estimate  fails. Thus  for any $k\in \N$, there 
are   $(v_{k},p_{k})\in\B_{1/k}(0,0)$ and $(x_k,\lm_k)\in S(v_{k},p_{k})\cap\B_{1/k}(\ox,\olm)$ satisfying
$$
\frac{\|x_k-\ox\|}{\|P_{\D_+}(v_{k})\|+\|p_{k}\|}\to\infty\;\mbox{ as }\;k\to\infty.
$$
Set $t_k:=\|x_k-\ox\|$ and hence obtain  $P_{\D_+}(v_{k})=o( t_k)$ and $p_{k}=o(t_k)$. By passing to a subsequence if necessary, we can assume that 
\begin{equation*}\label{xi}
\frac{x_k-\ox}{t_k}\to w\;\mbox{ as }\;k\to\infty\;\mbox{ with some }\;0\ne w\in\R^n.
\end{equation*}
Since $(x_k,\lm_k)\in S(v_{k},p_{k})$ and $(\ox,\olm)\in S(0,0)$, we conclude from \eqref{maps} that 
$$
v_k-\nabla_x L(x_k,\lm_k)=q_k\;\;\mbox{with}\;\;q_k\in N_\Th(x_k)\quad \mbox{and}\quad -\nabla_x L(\ox,\olm)\in N_\Th(\ox).
$$
It follows from these and  the reduction lemma for a 
polyhedral convex set from\eqref{rl1} that for all $k$ sufficiently large we have 
$$
v_k-\big(\nabla_x L(x_k,\lm_k) -\nabla_x L(\ox,\olm)\big)\in N_{K_\Th(\ox, -\nabla_x L(\ox,\olm))}(x_k-\ox).
$$
This tells us that $(x_k-\ox)/t_k\in K_\Th(\ox, -\nabla_x L(\ox,\olm))$ and thus $w\in K_\Th(\ox, -\nabla_x L(\ox,\olm))$. Since   
 $\Th$  is a polyhedral convex set, so is the critical cone $K_\Th(\ox, -\nabla_x L(\ox,\olm))$. Thus we get  the inclusion 
$$
N_{K_\Th(\ox, -\nabla_x L(\ox,\olm))}\big(\frac{x_k-\ox}{t_k}\big)\subset N_{K_\Th(\ox, -\nabla_x L(\ox,\olm))}(w),
$$
which in turn results in \eqref{pe1}. 
  It follows from the relationships  $P_{\D_+}=(I+N_{\D_+})^{-1}$ and    $P_{\D_+}(v_{k})=o( t_k)$ that $v_k+o(t_k)\in N_{\D_+}\big(o(t_k)\big)$,
  which yields  $v_k+o(t_k)\in \D_+^\bot$.
 Since $\D_+$ is a linear subspace, we get  $-v_k+o(t_k)\in \D_+^\bot$. This, combined with \eqref{pe1}, brings us to 
$$
o(t_k)-\big(\nabla_x L(x_k,\lm_k) -\nabla_x L(\ox,\olm)\big)\subset N_{K_\Th(\ox, -\nabla_x L(\ox,\olm))}(w)+\D_+^\bot.
$$
 Since we have 
\begin{eqnarray*}
  \nabla_x L(x_k,\lm_k)-\nabla_x L(\ox,\olm)& =& \nabla_x L(x_k,\olm)-\nabla_x L(\ox,\olm)+\nabla \Phi(x_k)^*(\lm_k-\olm)\nonumber\\
&=&\nabla^2_{xx} L(\ox,\olm)(x_k-\ox)+\nabla \Phi(\ox)^*(\lm_k-\olm) +o(t_k),
\end{eqnarray*}
we obtain 
\begin{equation*} 
\frac{o(t_k)}{t_k}-\nabla^2_{xx} L(\ox,\olm)(\frac{x_k-\ox}{t_k})-\nabla \Phi(\ox)^*(\frac{\lm_k-\olm}{t_k})\subset N_{K_\Th(\ox, -\nabla_x L(\ox,\olm))}(w)+\D_+^\bot.
\end{equation*}
This together with \eqref{pe5} and \eqref{pe6} ensures the existence of a sequence $\{b_k\b$ in $\B$ so that 
\begin{eqnarray*}
 \frac{o( t_k)}{t_k}-\nabla^2_{xx} L(\ox,\olm)\big(\frac{x_k-\ox}{t_k}\big) -\ell\big\|\frac{z_k-\oz}{t_k}-\nabla \Phi(\ox)w\big\| \nabla \Phi(\ox)^*b_k\nonumber\\
 \in N_{K_\Th(\ox, - \nabla_x L(\ox,\olm))}(w)+ \nabla \Phi(\ox)^* D(\sub g)(\oz,\olm)(\nabla \Phi(\ox)w)+\D_+^\bot.
 \end{eqnarray*}
 Since the sets on the right-hand side of this inclusion are polyhedral, their sum is a closed set. Thus,
 passing to the limit in the above inclusion tells us that 
\begin{equation} \label{pe7}
 0\in \nabla^2_{xx} L(\ox,\olm)w + N_{K_\Th(\ox, - \nabla_x L(\ox,\olm))}(w)+ \nabla \Phi(\ox)^* D(\sub g)(\oz,\olm)(\nabla \Phi(\ox)w)+\D_+^\bot.
\end{equation}
Similar to the proof of Theorem~\ref{prim1} (see the line before \eqref{pe4}), we can show that $w\in \D$, where $\D$ comes from \eqref{coned}, which leads us to 
\begin{equation}\label{pe8}
w\in K_\Th\big(\ox,-\nabla_x L(\ox,\olm) \big)\quad \mbox{and}\quad \nabla \Phi(\ox)w\in K_g(\Phi(\ox),\olm).
\end{equation}
 Observe also by \cite[Corollary~11.25(d)]{rw} that 
\begin{eqnarray}
 \D_+^\bot &=&\Big(K_\Th \big(\ox,-\nabla_x L(\ox,\olm) \big)^*\cap -K_\Th \big(\ox,-\nabla_x L(\ox,\olm) \big)^*\Big)\nonumber \\
 &&+\Big\{\nabla \Phi(\ox)^*u\Big|\;   u\in K_g(\Phi(\ox),\olm)^*\cap -K_g(\Phi(\ox),\olm)^*\Big\} \label{pe9}.
 \end{eqnarray}
We   proceed by justifying two claims:\\

{\bf Claim I.} {\em The following inclusion holds: $$\Big(K_\Th \big(\ox,-\nabla_x L(\ox,\olm) \big)^*\cap -K_\Th \big(\ox,-\nabla_x L(\ox,\olm) \big)^*  \Big)\subset N_{K_\Th(\ox, - \nabla_x L(\ox,\olm))}(w).$$}
 To prove this claim,  pick $ u\in K_\Th \big(\ox,-\nabla_x L(\ox,\olm) \big)^*\cap -K_\Th \big(\ox,-\nabla_x L(\ox,\olm) \big)^*$. It follows from the  first inclusion in  \eqref{pe8} that $\la  u,w\ra =0$. This 
 along with the fact that $K_\Th \big(\ox, - \nabla_x L(\ox,\olm) \big)$ is a convex cone  implies   via \cite[Proposition~2A.3]{dr} that 
 $$
 u \in N_{K_\Th(\ox, - \nabla_x L(\ox,\olm))}(w),
 $$
and hence proves the claimed inclusion.

{\bf Claim II.} {\em For any $\eta \in D(\sub g)(\oz,\olm)(\nabla \Phi(\ox)w)$ and any $ u\in K_g(\Phi(\ox),\olm)^*\cap -K_g(\Phi(\ox),\olm)^*$, we have 
$$ \nabla \Phi(\ox)^*( \eta +   u) \subset  \nabla \Phi(\ox)^* D(\sub g)(\oz,\olm)(\nabla \Phi(\ox)w).$$}
To verify this inclusion, let $\eta \in D(\sub g)(\oz,\olm)(\nabla \Phi(\ox)w)$ and $u\in K_g(\Phi(\ox),\olm)^*\cap -K_g(\Phi(\ox),\olm)^*$.
We conclude from   the second inclusion in  \eqref{pe8} that  $\la u,  \nabla \Phi(\ox)w\ra=0$.
It follows from $u\in K_g(\Phi(\ox),\olm)^* $ and Proposition~\ref{sop}(a) that $u\in K_{C_i}(\Phi(\ox),\olm_i)^*$ for all $i\in I(\Phi(\ox))$, where $ K_{C_i}(\Phi(\ox),\olm_i)$ is taken from \eqref{cc2}.
Combining these implies that 
$$
u\in N_{K_{C_i}(\Phi(\ox),\olm_i)}(\nabla \Phi(\ox)w)\quad \mbox{for all}\;\; i\in I(\Phi(\ox)).
$$
Moreover, we conclude from $\eta \in D(\sub g)(\oz,\olm)(\nabla \Phi(\ox)w)$  and Proposition~\ref{sop}(c) that 
$$
\eta-A_i(\nabla \Phi(\ox)w)\in N_{K_{C_i}(\Phi(\ox),\olm_i)}(\nabla \Phi(\ox)w)\quad \mbox{for all}\;\; i\in {\mathfrak J}(\nabla \Phi(\ox)w).
$$
 These inclusions as well as ${\mathfrak J}(\nabla \Phi(\ox)w)\subset I(\Phi(\ox))$ result in 
 $$
 \eta+u-A_i(\nabla \Phi(\ox)w)\in N_{K_{C_i}(\Phi(\ox),\olm)}(\nabla \Phi(\ox)w)\quad \mbox{for all}\;\; i\in {\mathfrak J}(\nabla \Phi(\ox)w),
 $$  
implying that $\eta+u\in  D(\sub g)(\oz,\olm)(\nabla \Phi(\ox)w)$ by Proposition~\ref{sop}(c). This  justifies Claim II.

Using Claims I and II together with \eqref{pe7} and \eqref{pe9}, we arrive at the inclusion 
$$
 0\in \nabla^2_{xx} L(\ox,\olm)w + N_{K_\Th(\ox, - \nabla_x L(\ox,\olm))}(w)+ \nabla \Phi(\ox)^* D(\sub g)(\oz,\olm)(\nabla \Phi(\ox)w),
$$
a contradiction with  $\olm$ being a noncritical multiplier for \eqref{vs} since $w\neq 0$. This ends the proof.
 \end{proof}
 
 As pointed out in \cite[page~3322]{fis12}, the right-hand sides of the estimates \eqref{semi2} and \eqref{semi3} involve the sets $\D$ and $\D_+$, respectively, which are 
 defined at the unknown solution $(\ox,\olm)$, and so are not computable if we want to use them in algorithms. The purpose of establishing such estimates 
 is   only for the local convergence analysis of the quasi-Newton SQP method for \eqref{comp}.
 
\section{Primal-Dual Superlinear Convergence of  SQP Methods}\sce  \label{sect05} 

This section is devoted to the local convergence analysis of the basic SQP method for the composite optimization problem \eqref{comp}. 
To this end, we are going to apply \cite[Theorem~3.2]{is14} in which the  superlinear convergence of the Newton method  was established for generalized equations under two assumptions: 1) semistability 
and 2) hemistability; see \cite[page~140]{is14} for more detail. Since  the KKT system \eqref{vs} can be equivalently formulated as the generalized equation 
\begin{equation}\label{ge2}
 \begin{bmatrix}
0\\
0
\end{bmatrix}
 \in\begin{bmatrix}
\nabla_x L(x,\lm)\\
-\Phi(x)
\end{bmatrix}+\begin{bmatrix}
N_\Th(x)\\
(\sub g)^{-1}(\lm)
\end{bmatrix},
\end{equation}
we should find conditions that ensure the validity of the latter assumptions for the generalized equation \eqref{ge2}. The semistability of \eqref{ge2} (cf. \cite[Definition~1.29]{is14}) amounts to 
the isolated calmness of the solution mapping $S$ from \eqref{maps}, which by Theorem~\ref{smrso} can be ensured under the second-order sufficient condition \eqref{sosc} and the uniqueness of Lagrange multipliers.
To analyze the second assumption, we first recall its definition, adapted for \eqref{ge2}: A solution $(\ox,\olm)$  to the generalized equation \eqref{ge2} is called {\em hemistable} if for any $(u,\mu)\in \R^n\times \R^m$ sufficiently close to 
$(\ox,\olm)$, the generalized equation 
\begin{equation}\label{kktsub4}
 \begin{bmatrix}
0\\
0
\end{bmatrix}
  \in 
 \begin{bmatrix}
\nabla_{x}L(u,\mu)\\
-\Phi(u)
\end{bmatrix}
 +
 \begin{bmatrix}
 \nabla_{xx}^2L(u,\mu)& \nabla \Phi(u)^* \\
-\nabla \Phi(u)&0
\end{bmatrix}
 \begin{bmatrix}
x-  u\\
\lm-\mu
\end{bmatrix}
+\begin{bmatrix}
N_\Th(x)\\
(\sub g)^{-1}(\mu)
\end{bmatrix},
\end{equation}
has a solution $(x,\lm)$ that converges to  $(\ox,\olm)$ as $(u,\mu)\to (\ox,\olm)$. It is not hard to see that \eqref{kktsub4} is, indeed, the KKT system of the subproblem \eqref{subs1}
with $(x_k,\lm_k):=(u,\mu)$ and $H_k$ taken from \eqref{hk}. We are going to show that the hemistability of the solution $(\ox,\olm)$ to \eqref{vs} can be also ensured by  the second-order sufficient condition \eqref{sosc} and the uniqueness of Lagrange multipliers.
To this end, consider a parameter space $\R^d$,  the functions $f:\R^n\times \R^d\to \R$, and $\Psi:\R^n\times \R^d\to \R^m$   that 
  are   continuously differentiable. Define now the parametrized composite problem   
\begin{equation}\label{coop2}
\mbox{minimize}\;\;\;f(x,p)+g(\Psi(x,p))\quad\mbox{subject to  }\;\;x\in \Th,
\end{equation}
where $g$ and $\Th$ are taken from \eqref{comp}, namely $g:\R^m\to \oR$ is  CPLQ and $\Th$ is a polyhedral convex set in $\R^n$.
The following result is an extension of \cite[Theorem~1.21]{is14}, which was established for NLPs. While the proof 
uses a similar argument, it requires some small adjustments for the composite problem \eqref{comp}. So we provide a proof for the readers' convenience. 
Note that while constraint qualifications are often utilized for the subdifferetial calculus in variational analysis, the imposed constraint qualification in the following result is to ensure the Aubin property of the constraint mapping \eqref{conm}. 
\begin{Proposition}[existence of local minimizers of parametrized problems] \label{los}
Let $\op\in \R^d$ and $\ox\in \Th$,  let  $\ox$ be a strict local  minimizer of \eqref{coop2} for $p=\op$, and let 
the basic constraint qualification 
 \begin{equation}\label{rcq2}
 -\nabla_x \Psi(\ox,\op)^* u\in N_\Th(\ox), \;\; u\in N_{\ss\dom g}(\Psi(\ox,\op))\implies u=0
 \end{equation}
hold. Then for any $p\in \R^d$ sufficiently close to $\op$, the problem \eqref{coop2} admits a local minimizer $x_p$   such that $x_p\to \ox$  as $p\to \op$.
\end{Proposition}
\begin{proof} By assumptions, we can find a constant $\ve>0$ such that $\ox$ is the strict minimizer of the problem 
\begin{equation}\label{pro0}
\mbox{minimize}\;\;\;f(x,\op)+g(\Psi(x,\op))\quad\mbox{subject to  }\;\;x\in \Th\cap   \B_\ve(\ox).
\end{equation}
Pick a parameter $p\in \R^d$ and define the set-valued mapping $\Gamma:\R^d\tto \R^n$ by 
\begin{equation}\label{conm}
\Gamma(p):=\big\{x\in \Th\big|\; \Psi(x,p)\in \dom g\big\}.
\end{equation}
According to \cite[Example~9.51]{rw}, the mapping $\Gamma$ enjoys the Aubin property around $(\op,\ox)\in \gph \Gamma$, meaning that there exist
neighborhoods $U$ of $\op$ and $V$ of $\ox$ and a constant $\ell\ge 0$ for which we have 
\begin{equation}\label{aubin}
\Gamma(p)\cap V\subset \Gamma(p')+\ell\|p-p'\|\B \quad\mbox{for all}\; p,p'\in U.
\end{equation}
Shrinking the neighborhoods $U$ and $V$ if necessary, we conclude from \eqref{aubin} that $\Gamma(p)\cap \B_\ve(\ox)\neq \emptyset$ for all $p\in U$. Pick  $p\in U$ and consider  the problem 
\begin{equation}\label{pro1}
\mbox{minimize}\;\;\;f(x,p)+g(\Psi(x,p))\quad\mbox{subject to  }\;\;x\in \Th\cap   \B_\ve(\ox).
\end{equation}
Since  $\Gamma(p)\cap \B_\ve(\ox)\neq \emptyset$, the classical Weierstrass theorem implies that problem \eqref{pro1} admits a minimizer $x_p$. 
We claim now that $x_p\to \ox$ as $p\to \op$. Suppose by contradiction that this convergence fails, meaning that there exists a sequence $p_k\to \op$ 
for which the minimizers $x_{p_k}$ of \eqref{pro1} for $p=p_k$ do not converge to $\ox$. Since $x_{p_k}\in  \B_\ve(\ox)$, by passing to a subsequence if necessary, we can assume 
that $x_{p_k}\to u$ for some $u\in \B_\ve(\ox)$ with $u\neq \ox$. It follows from  $\ox\in \Gamma(\op)\cap V$ and  \eqref{aubin} that for any sufficiently large $k$, we can find $y_{p_k}\in \Gamma(p_k)$
such that 
$$
\|y_{p_k}-\ox\|\le \ell\|p_k-\op\|. 
$$
This tells us that $y_{p_k}\to \ox$ as $k\to \infty$ and so $y_{p_k}\in \B_\ve(\ox)$ for all $k$ sufficiently large. Since $x_{p_k}$ is a minimizer of \eqref{pro1}, we get 
$$
f(x_{p_k},p_k)+g(\Psi(x_{p_k},p_k)) \le f(y_{p_k},p_k)+g(\Psi(y_{p_k},p_k)).
$$
Since $g$ is continuous relative to its domain (cf. \cite[Proposition~10.21]{rw}) and since $\Psi(y_{p_k},p_k)\in \dom g$ and $\Psi(x_{p_k},p_k)\in \dom g$, passing to the limit brings us to 
$$
f(u,\op)+g(\Psi(u,\op))\le f(\ox,\op)+g(\Psi(\ox,\op)).
$$
Remember that $u\in \B_\ve(\ox)$ with $u\neq \ox$. This together with the inequality above tells us that $u$ is a minimizer of \eqref{pro0}, a contradiction. This proves the claim that $x_p\to \ox$ as $p\to \op$
and hence completes the proof.
\end{proof}

After this presentation, we are now ready to prove   the hemistability of a solution $(\ox,\olm)$ to \eqref{vs} under the second-order sufficient condition and 
the uniqueness of Lagrange multipliers. 
\begin{Proposition}[solvability of subproblems in the basic SQP method] \label{solva}
Let $(\ox,\olm)$ be a solution to the KKT system \eqref{vs} and let 
the second-order sufficient condition \eqref{sosc} be satisfied at $(\ox, \olm)$ and $\Lambda(\ox)=\{\olm\}$.  Then $(\ox,\olm)$ is a hemistable solution to the KKT system \eqref{vs}.
\end{Proposition}
\begin{proof}  For any $p:=(u,\mu)\in \R^n\times \R^m$ and $x\in \R^n$,  define   
the functions $f:\R^n\times \R^n\times \R^m\to \R$ and $\Psi:\R^n\times \R^n\times \R^m\to \R^m$, respectively,  by 
\begin{equation}\label{psph}
\begin{cases}
f(x,p)=\ph(u)+ \la \nabla \ph(u), x- u\ra+ \sm\la \nabla_{xx}^2L( u,\mu)(x-u), x-u \ra,\\
\Psi(x,p)= \Phi(u)+\nabla \Phi(u)(x-u ).
\end{cases}
\end{equation}
So we can view the SQP subproblem \eqref{subs1} with $(x_k,\lm_k):=(u,\mu)$ and $H_k$ taken from \eqref{hk} as the parametrized composite optimization problem 
\begin{equation}\label{comp2}
\mbox{minimize}\;\;\;f(x,p)+g(\Psi(x,p))\quad\mbox{subject to  }\;\;x\in \Th,
\end{equation}
 with $f$ and $\Psi$ defined by \eqref{psph}. Set $\op:=(\ox,\olm)$ and observe that 
\begin{equation}\label{pe12}
\begin{cases}
\nabla_x f(x,\op)=\nabla \ph(\ox)+\nabla_{xx}^2L( \ox,\olm)(x- \ox), \;\; \nabla^2_{xx} f(x,\op)= \nabla_{xx}^2L(\ox,\olm),\\
\Psi(\ox,\op)= \Phi(\ox), \;\; \nabla_x \Psi(x,\op)=\nabla \Phi(\ox).
\end{cases}
\end{equation}
These equalities tell us that  the KKT system of \eqref{comp2} for  $p=\op$ is the generalized equation 
\begin{equation}\label{gekk}
0\in \nabla \ph(\ox) + \nabla_{xx}^2L( \ox,\olm)(x- \ox)+ \nabla \Phi(\ox)^*\lm+N_\Th(x),\;\; \lm\in \sub g\big(\Phi(\ox)+\nabla \Phi(\ox)(x-\ox)\big).
\end{equation}
It is not hard to see that  $(\ox,\olm)$ is a solution to this KKT system, implying  that $\ox$ is  a stationary point for  \eqref{comp2} associated with $p=\op$ and  
that $\olm$ is a Lagrange multiplier associated with $\ox$ for the latter problem. Define the Lagrangian of \eqref{comp2} by $\L(x,p,\lm):=f(x,p)+\la\lm, \Psi(x,p)\ra$ and deduce from \eqref{pe12} that 
\begin{equation}\label{pe15}
\nabla_x\L(\ox,\op,\olm)=\nabla_x L(\ox,\olm)\quad \mbox{and}\quad \nabla^2_{xx}\L(\ox,\op,\olm)=\nabla^2_{xx} L(\ox,\olm).
\end{equation}
To simplify the proof, we are going to break it down into the following steps: 

{\bf Step 1.} {\em The basic constraint qualification \eqref{rcq2} holds for the parametrized   problem \eqref{comp2} at $(\ox,\op)$.}

To prove this claim, we conclude from $\Lambda(\ox)=\{\olm\}$ and Proposition~\ref{unic} that the dual condition \eqref{dual} holds.
This together with  \eqref{pe14} tells us that the basic constraint qualification \eqref{rcq} is satisfied. Appealing now to  \eqref{pe12} gives us \eqref{rcq2}.

{\bf Step 2.} {\em The set of Lagrange multipliers  of \eqref{comp2}   associated with $(\ox,\op)$ is $\{\olm\}$.}

To justify this claim,  we deduce from  \eqref{pe12} and the generalized equation \eqref{gekk}  that 
 the Lagrange multiplier set associated with $(\ox,\op)$ for \eqref{comp2}   coincides with that of the composite problem \eqref{comp}.
Since the latter is $\{\olm\}$, we finish the proof of this step.

{\bf Step 3.} {\em $\ox$ is a strict local minimizers of the parametrized   problem \eqref{comp2}   for $p=\op $.}

To verify this step, we show that the second-order sufficient condition of the type \eqref{sosc} holds for the parametrized problem 
\eqref{comp2} at $\big((\ox,\op),\olm)$. To this end, pick $w\in \R^n $ and  conclude from \eqref{pe15} 
and \eqref{pe12} that 
\begin{equation*}
 \big\la  \nabla^2_{xx}\L(\ox,\op,\olm)w,w\big\ra+\d^2 g\big(\Psi(\ox,\op),\olm\big)\big(\nabla_x \Psi(\ox,\op)w\big)=\big\la\nabla^2_{xx}L(\ox,\olm)w,w\big\ra+\d^2 g\big(\Phi(\bar x),\olm\big)\big(\nabla \Phi(\ox)w\big),
\end{equation*}
and that 
$$
K_\Th\big(\ox,- \nabla_x\L(\ox,\op,\olm)\big)=K_\Th\big(\ox,- \nabla_xL(\ox,\olm)\big)\quad \mbox{and}\quad \;K_g(\Psi(\ox,\op),\olm)=K_g(\Phi(\ox),\olm).
$$
Since the second-order sufficient condition \eqref{sosc} holds  at $(\ox,\olm)$, the above equalities confirm that the  
second-order sufficient condition of the type \eqref{sosc} holds for \eqref{comp2} at $\big((\ox,\op),\olm)$. This along with Step 1  and Proposition~\ref{sooc}(b) 
proves that $\ox$ is a strict local minimizer of \eqref{comp2}   for $p=\op $.

Appealing now to Proposition~\ref{los} and Steps 1 and  3, we conclude that for any $p$ sufficiently close to $\op$ the parametrized problem \eqref{comp2}
admits a local minimizer $x_p$ that $x_p\to \ox$ as $(u,\mu)=p\to \op=(\ox,\olm)$.  

{\bf Step 4.} {\em For any $p$ sufficiently close to $\op$, there exists a Lagrange multiplier $\lm_p$ associated with the local minimizer $x_p$ of the parametrized problem \eqref{comp2} that 
$\lm_p\to \olm$ as $p\to \op$.}

To furnish this step, by  Step 1, we can find a neighborhood of $\op$ such that for any $p$ in this neighborhood the basic constraint qualification 
$$
 -\nabla_x \Psi(x_p,p)^* u\in N_\Th(x_p), \;\; u\in N_{\ss\dom g}(\Psi(x_p,p))\implies u=0
$$
fulfills. This, combined with \cite[Example~10.8]{rw}, ensures the existence of a Lagrange multiplier $\lm_p$ associated with the local minimizer $x_p$ of the parametrized problem \eqref{comp2}.

By Step 1, the  basic constraint qualification \eqref{rcq2} holds for  \eqref{comp2} at $(\ox,\op)$, which tells us that the Lagrange multipliers $\lm_p$ are uniformly bounded whenever 
$p$ is chosen sufficiently close to $\op$. Indeed, if this fails, we find  sequences $\{p_k\b$, converging to $\op$,   and $\{\lm_{p_k}\b$, which is unbounded. 
Recall that $\lm_{p_k}$ is a Lagrange multiplier  associated with the local minimizer $x_{p_k}$ of the parametrized problem \eqref{comp2} with $p=p_k$.
This tells us that for each $k$ we have 
\begin{equation}\label{dc0}
0\in \nabla_x f(x_{p_k},p_k)+\nabla_x \Psi(x_{p_k},p_k)^*\lm_{p_k}+ N_\Th(x_{p_k}),\;\;\lm_{p_k} \in \sub g\big(\Psi(x_{p_k},{p_k})\big).
\end{equation}
Since $\{\lm_{p_k}\b$ is unbounded, we can assume by passing to a subsequence if necessary that $\lm_{p_k}/\|\lm_{p_k}\| \to  \eta$  as $k\to \infty$ for some $\eta\in \R^m\setminus \{0\}$.
Dividing both sides of \eqref{dc0} by $\| \lm_{p_k}\| $ and then passing to the limit, we arrive at
$$
 -\nabla_x \Psi(\ox,\op)^* \eta\in N_\Th(\ox), \;\; \eta\in N_{\ss\dom g}(\Psi(\ox,\op)),
$$
a contradiction with  \eqref{rcq2}  since $\eta\neq 0$. 
Since the set of Lagrange multipliers  of \eqref{comp2} associated with $p=\op$ is $\{\olm\}$ (Step 2), we arrive at $\lm_p\to \olm$ as $p\to \op$.
Clearly, for any such a $p$, the pair  $(x_p,\lm_p)$ is a solution to the KKT system of \eqref{comp2}, namely the generalized equation \eqref{kktsub4}.
Because we have $(x_p,\lm_p)\to (\ox,\olm)$ as $(u,\mu)=p\to \op=(\ox,\olm)$,  $(\ox,\olm)$ is a hemistable solution to the KKT system \eqref{vs}.
\end{proof}

Note that Proposition~\ref{solva} is an extension of \cite[Proposition~6.3]{bo94}, which was established a similar conclusion for NLPs; see also  \cite[Theorem~5.2]{mms3} for  
a similar result for parabolically regular constrained optimization problems. 

Recall that the generic SQP method for the composite problem \eqref{comp} is given as follows: 
\begin{Algorithm} [generic SQP method]\label{sqpal} {\rm Choose $(x_k,\lm_k)\in \R^n\times \R^m$ and set $k=0$.
\begin{itemize}[noitemsep]
\item [(1)] If $(x_k,\lm_k)$ satisfies the KKT system \eqref{vs}, then stop.
\item [(2)] Choose an $n\times n$ symmetric matrix $H_k$ and compute $(x_{k+1},\lm_{k+1})$
as a solution to the KKT system of the subproblem \eqref{subs1}, which can be described by the generalized equation
\begin{equation}\label{kktsub}
 \begin{bmatrix}
0\\
0
\end{bmatrix}
  \in 
 \begin{bmatrix}
\nabla_{x}L( x_k,\lm_k)\\
-\Phi(x_k)
\end{bmatrix}
 +
 \begin{bmatrix}
H_k& \nabla \Phi(x_k)^* \\
-\nabla \Phi(x_k)&0
\end{bmatrix}
 \begin{bmatrix}
x-x_k\\
\lm- \lm_k
\end{bmatrix}
+\begin{bmatrix}
N_\Th(x)\\
(\sub g)^{-1}(\lm)
\end{bmatrix}.
\end{equation}
\item [(3)] Increase $k$ by $1$ and then go back to Step (1).
\end{itemize}}
\end{Algorithm}

Now we are ready to present the primal-dual superlinear convergence of the basic SQP method.

\begin{Theorem}[primal-dual superlinear convergence of the basic SQP method]\label{supsqp} Assume that  $(\ox,\olm)$ is a solution to the KKT system \eqref{vs},  that  the second-order sufficient condition \eqref{sosc} is satisfied at $(\ox,\olm)$, and that   
 $\Lambda(\ox)=\{\olm\}$. Then there exists a positive constant $\dd$ such that for any starting point $(x_0,\lm_0)\in\R^n\times\R^m$ sufficiently close to $(\ox,\olm)$,
we can find a sequence $\{(x_k,\lm_k)\}\subset\R^n\times\R^m$, generated by Algorithm~{\rm\ref{sqpal}} with $H_k$ taken from \eqref{hk}, satisfying
\begin{equation}\label{lc1}
\|(x_{k+1}-x_k,\lm_{k+1}-\lm_k)\|\le\dd.
\end{equation}
Moreover, every such a sequence  converges to $(\ox,\olm)$, and the rate of convergence is superlinear.
\end{Theorem}
\begin{proof} As pointed out earlier in this section,  the superlinear convergence of the generalized equation \eqref{ge2}
can be ensured via   \cite[Theorem~3.2]{is14} under the semistability and hemistability of $(\ox,\olm)$. Remember that 
the former   amounts to the isolated calmness of the solution mapping $S$, which is satisfied by Theorem~\ref{smrso} under the imposed 
assumptions. The hemistability  of $(\ox,\olm)$ comes from Proposition~\ref{solva}. Appealing now to  \cite[Theorem~3.2]{is14} justifies the claimed conclusions.
\end{proof}

Note that the primal-dual superlinear convergence of the basic SQP method for \eqref{comp} with $\Th=\R^n$
was established recently in \cite[Theorem~7.3]{be} under the strong second-order sufficient condition, the nondegeneracy condition, and the strict complementary condition,
which are strictly stronger than the assumptions utilized in Theorem~\ref{supsqp}. In fact, the assumptions used in \cite[Theorem~7.3]{be} result in the strong metric regularity of 
the solution mapping $S$, which means that its inverse mapping, namely $G$ from \eqref{GKKT}, admits a
single-valued Lipschitzian graphical localization (see \cite[page~4]{dr}). However, our assumptions in Theorem~\ref{supsqp} imply via Theorem~\ref{smrso} and $G=S^{-1}$ that the    mapping $G$ is strongly metrically subregular, which is strictly 
weaker than the latter strong metric regularity. 

Note also that the imposed assumptions in Theorem~\ref{smrso} do not guarantee the uniqueness of minimizers of the subproblems \eqref{subs1}. 
So the localization condition \eqref{lc1} is required to filter out those minimizers of \eqref{subs1} that are not   sufficiently close  to $\ox$; see \cite[Examples~5.1 and 5.2]{is15} for a detailed discussion about the importance 
of \eqref{lc1}.

\section{Primal Superlinear Convergence of Quasi-Newton   SQP Methods}\sce \label{sect06}

This section aims to present our main results in this paper in which we  characterize   the primal superlinear convergence of the quasi-Newton SQP method for the composite problem \eqref{comp}
via the Dennis-Mor\'e condition \eqref{dmc}. It is worth mentioning that in general the superlinear convergence of a primal-dual sequence 
does not yield that of the primal part of the sequence; see \cite[Exercise~14.8]{bgls}. Since we do not   have a primal-dual convergence  rate  for the 
quasi-Newton SQP method, achieving the primal superlinear convergence  for the latter method is of great importance. Our first result 
provides a characterization of this primal superlinear convergence under the second-order sufficient condition.

\begin{Theorem}[characterization of primal superlinear convergence]\label{psqp}
 Let $(\ox,\olm)$ be a solution to the KKT system \eqref{vs}, let  $\{H_k\b$ be a sequence of  $n\times n$ symmetric matrices 
 and let $\{(x_k,\lm_k)\b$ be constructed via Algorithm~{\rm\ref{sqpal}}.
 Assume further that the sequence $\{(x_k,\lm_k)\b$ converges to  $(\ox,\olm)$ as $k\to \infty$. Then the following conditions hold:
\begin{itemize}[noitemsep,topsep=0pt]
\item[\rm{(a)}]  if $\Th=\R^n$ in \eqref{comp} and the second-order sufficient condition \eqref{sosc} 
holds at $(\ox,\olm)$ and if the 
Dennis-Mor\'e condition \eqref{dmc}   is satisfied, then the rate of convergence of the primal sequence $\{x_k\b$ is superlinear;
\item[\rm{(b)}] if the rate of convergence of the primal sequence $\{x_k\b$ is superlinear, then the Dennis-Mor\'e condition \eqref{dmc}
is satisfied.

\end{itemize}
\end{Theorem}
\begin{proof} 
Remember that  $(x_{k+1},\lm_{k+1})$ is a solution to  the generalized equation \eqref{kktsub}. This gives us  
\begin{equation}\label{fn88}
\begin{cases}
0\in \nabla_{x}L( x_{k},\lm_k) + H_k(x_{k+1}-x_k) + \nabla \Phi(x_{k})^*( \lm_{k+1}-\lm_{k})+N_\Th(x_{k+1})\;\;\mbox{and}\\
\lm_{k+1}    \in \sub g\big ( \Phi(x_{k})+\nabla \Phi(x_{k})( x_{k+1}-x_{k})\big).
 \end{cases}
\end{equation}
Since   $L$ and $\Phi$ are twice continuously differentiable around $\ox$, we get
\begin{eqnarray}\label{fn87}
\nabla_{x}L( x_{k+1},\lm_{k+1})&=& \nabla_{x}L( x_{k+1},\lm_k)+ \nabla \Phi(x_{k+1})^*( \lm_{k+1}-\lm_{k}) \nonumber\\
&=&  \nabla_{x}L( x_{k},\lm_k) + \nabla^2_{xx}L( x_k,\lm_k)(x_{k+1}-x_k) \nonumber \\
&&+ \nabla \Phi(x_{k})^*( \lm_{k+1}-\lm_{k})+o(\|x_{k+1}-x_k\|),
\end{eqnarray}
and  
\begin{equation}\label{fn34}
\Phi(x_{k+1})= \Phi(x_{k})+\nabla \Phi(x_{k})( x_{k+1}-x_{k})+ o(\|x_{k+1}-x_k\|).
\end{equation}
Set 
$u_{k+1}:=\big( \nabla^2_{xx}L( x_k,\lm_k)-H_k\big)(x_{k+1}-x_k) $,  $p_{k+1}:= \Phi(x_{k})+\nabla \Phi(x_{k})( x_{k+1}-x_{k}) -\Phi(x_{k+1})$ and 
$$
q_{k+1}:=\nabla_{x}L( x_{k+1},\lm_{k+1})-  \nabla_{x}L( x_{k},\lm_k) - \nabla^2_{xx}L( x_k,\lm_k)(x_{k+1}-x_k)- \nabla \Phi(x_{k})^*( \lm_{k+1}-\lm_{k}).
$$
It follows from \eqref{fn87} and \eqref{fn34}, respectively, that $p_{k+1}=o(\|x_{k+1}-x_k\|)$ and $q_{k+1}=o(\|x_{k+1}-x_k\|)$. Also, observe that    \eqref{fn88} can be equivalently rewritten as 
\begin{equation}\label{pe20}
v_{k+1}:=u_{k+1}+q_{k+1}\in \nabla_{x}L( x_{k+1},\lm_{k+1}) +N_\Th(x_{k+1}), \;\;\lm_{k+1}    \in \sub g\big ( \Phi(x_{k+1})+  p_{k+1}\big).
\end{equation}

After these presentations, we begin to prove (a). It follows from \eqref{sosc} and Theorem~\ref{prim1} that there are some neighborhoods
$U$ of $(0,0)\in \R^n\times \R^m$ and $V$ of $(\ox,\olm)$ and a constant $\kappa\ge 0$ such that 
for any $(v,p)\in U$ and any  $(x,\lm)\in  S(v,p)\cap V$ the estimate \eqref{semi2} holds, where the solution mapping $S$ comes from \eqref{maps}.
By $(x_k,\lm_k)\to (\ox,\olm)$, we can assume without loss of generality that $(x_k,\lm_k)\in V$
for all $k\in \N$. Moreover, by $\Th=\R^n$, we get $N_\Th(x_{k+1})=\{0\}$ in  \eqref{fn88}. The latter along with \eqref{fn88}, $\nabla_x L(\ox,\olm)=0$,  and $(x_k,\lm_k)\to (\ox,\olm)$ yields 
$H_k(x_{k+1}-x_k) \to 0$ as $k\to \infty$ and so  $u_{k+1}\to 0$ as $k\to \infty$. Thus we get 
$v_{k+1}\to 0$ and $p_{k+1}\to 0$ as $k\to \infty$. 
Again we can assume with no harm that $(v_{k+1},p_{k+1})\in U$  and  $( x_{k+1},\lm_{k+1})\in S( v_{k+1},p_{k+1})\cap V$ for all $k\in \N$. 
 Appealing now to \eqref{semi2} and the Dennis-Mor\'e condition \eqref{dmc} yields the estimates 
\begin{eqnarray*}
\|x_{k+1}-\ox\| &\le & \kappa\big(\|P_\D(v_{k+1})\|+\|p_{k+1}\|\big)\\
&\le & \kappa\big(\|P_\D(u_{k+1})\|+ o(\|x_{k+1}-x_k\|)\big)= o(\|x_{k+1}-x_k\|),
\end{eqnarray*}
which in turn    imply that 
\begin{equation*}
\|x_{k+1}-\ox\|=o(\|x_{k+1}-\ox\|+\|x_{k}-\ox\|).
\end{equation*}
Set $\al_k:= \dfrac{\|x_{k+1}-\ox\|}{\|x_{k+1}-\ox\|+\|x_{k}-\ox\|}$ and observe that 
\begin{equation}\label{fn90}
\frac{\|x_{k+1}-\ox\|}{\|x_{k}-\ox\|}= \al_k\cdot \big(\frac {\|x_{k+1}-\ox\|}{\|x_{k}-\ox\|}+1\big).
\end{equation}
This implies that 
$$
\frac{\|x_{k+1}-\ox\|}{\|x_{k}-\ox\|}=\frac{\al_k}{1-\al_k}\to 0\;\;\mbox{as}\;\;k\to \infty,
$$
 and hence proves the primal superlinear convergence of $\{x_k\b$,  claimed in (a).
 
 Turning to (b), assume that the rate of convergence of $\{x_k\b$ is superlinear, meaning that  $\|x_{k+1}-\ox\|=o(\|x_{k}-\ox\|)$ as  $k\to \infty$.
 Thus we can assume without loss of generality that 
$\|x_{k+1}-\ox\|\le \frac{1}{2}\|x_k-\ox\|$ for all $k$ sufficiently large. This implies that  
$$
\|x_k-\ox\|\le \|x_{k+1}-x_k\|+\|x_{k+1}-\ox\|\le \|x_{k+1}-x_k\|+ \frac{1}{2}\|x_k-\ox\|,
$$
which subsequently brings us to 
$$
\|x_k-\ox\|\le  2 \|x_{k+1}-x_k\|
$$
for all $k$ sufficiently large. Combining these, we get 
\begin{equation}\label{gl06}
\|x_{k+1}-\ox\|=o(\|x_{k+1}-x_k\|) \quad \mbox{as}\quad  k\to \infty.
\end{equation}
Since $x_k\to \ox$ as $k\to \infty$, we deduce from Proposition~\ref{ouli}(a) that there exists a constant $\ell\ge 0$ such that 
$$
  \sub g\big ( \Phi(x_{k+1})+  p_{k+1}\big)\subset \sub g(\Phi(\ox))+\ell\| \Phi(x_{k+1})+  p_{k+1}-\Phi(\ox)\|\B
$$
for all $k$ sufficiently large. By \eqref{pe20}, we have  $\lm_{k+1}    \in \sub g\big ( \Phi(x_{k+1})+  p_{k+1}\big)$.
This together with the inclusion above, \eqref{gl06}, the definition of $p_{k+1}$ implies that 
$$
\lm_{k+1}   + o(\|x_{k+1}-x_k\|) \in \sub g(\Phi(\ox)).
$$
Since $ \sub g(\Phi(\ox))$ is a polyhedral convex set, we conclude from \cite[Exercise~6.47]{rw} that 
$$\big( \sub g(\Phi(\ox))-\olm\big)\cap   {\cal O}=T_{ \ss \sub g(\Phi(\ox))}(\olm)\cap {\cal O}$$ 
for some neighborhood ${\cal O}$ of $0$  in $\R^m$.
This tells us that for all $k $ sufficiently large we get 
\begin{equation}\label{pe21}
\lm_{k+1} -\olm  + o(\|x_{k+1}-x_k\|) \in \big( \sub g(\Phi(\ox))-\olm\big)\cap {\cal O}\subset T_{\ss  \sub g(\Phi(\ox))}(\olm)=N_{\ss  \sub g(\Phi(\ox))}(\olm)^*=K_g(\Phi(\ox),\olm)^*,
\end{equation}
where the last equality results from \cite[Theorem~13.14]{rw}.  Because $x_k\to \ox$ as $k\to \infty$ and $\Th$ is polyhedral, we have the inclusion $N_\Th(x_k)\subset N_\Th(\ox)$
for all $k$ sufficiently large. This, combined with \eqref{pe20}, leads us to 
\begin{eqnarray}
v_{k+1}&\in&  \nabla_{x}L( x_{k+1},\lm_{k+1}) +N_\Th(x_{k+1})\nonumber\\
&\subset & \nabla_{x}L( x_{k+1},\olm)+\nabla \Phi(x_{k+1})^*(\lm_{k+1}-\olm) +N_\Th(\ox) \nonumber\\
&=& \nabla_{x}L( \ox,\olm) +  \nabla^2_{xx}L( \ox,\olm)(x_{k+1}-\ox)+ \nabla \Phi(\ox)^*(\lm_{k+1}-\olm)  \nonumber\\
&&+\big(\nabla^2 \Phi(\ox)(x_{k+1}-\ox)\big)^*(\lm_{k+1}-\olm)+o(\|x_{k+1}-\ox\|)+N_\Th(\ox) \nonumber\\
&=&  \nabla_{x}L( \ox,\olm)+ \nabla \Phi(\ox)^*(\lm_{k+1}-\olm)+ o(\|x_{k+1}-x_k\|)+N_\Th(\ox)\label{pe22},
\end{eqnarray}
where the last equality comes from \eqref{gl06}. Since both sets on the right-hand side of \eqref{coned}  are polyhedral, we have by \cite[Corollary~11.25(d)]{rw} that 
\begin{eqnarray*}
\D^*&=&K_\Th\big(\ox,-\nabla_x L(\ox,\olm)\big)^*+\Big(  \big\{w\in \R^n\big|\; \nabla \Phi(\ox)w\in K_g(\Phi(\ox),\olm)\big\}\Big)^*\\
&= & \Big(T_\Th(\ox)\cap [ \nabla_x L(\ox,\olm)]^\bot\Big)^*+ \big\{\nabla \Phi(\ox)^*u \big|\; u\in K_g(\Phi(\ox),\olm)^*\big\}\\
&= & N_\Th(\ox)+[ \nabla_x L(\ox,\olm)]+ \big\{\nabla \Phi(\ox)^*u \big|\; u\in K_g(\Phi(\ox),\olm)^*\big\}.
\end{eqnarray*}
This, \eqref{pe21}, and  \eqref{pe22} yield the inclusion 
$$
v_{k+1}+o(\|x_{k+1}-x_k\|)\in \D^*,
$$
which in turn results in 
$$
 P_\D\big(v_{k+1}+o(\|x_{k+1}-x_k\|)\big)=0,
$$ 
since $\D$ is a convex cone. Remember that $v_{k+1}=u_{k+1}+q_{k+1}$ with $q_{k+1}=o(\|x_{k+1}-x_k\|)$.  Thus we get
\begin{eqnarray*}
\|P_\D\big(u_{k+1}\big)\|&=&\| P_\D\big(u_{k+1}\big)-P_\D\big(v_{k+1}+o(\|x_{k+1}-x_k\|)\big)\|\\
&\le&  o(\|x_{k+1}-x_k\|)+ \|q_{k+1}\|=o(\|x_{k+1}-x_k\|).
\end{eqnarray*}
This implies that $P_\D\big(u_{k+1}\big)=o(\|x_{k+1}-x_k\|)$. Combining this and the definition of $u_{k+1}$ ensures the Dennis-Mor\'e condition \eqref{dmc} 
and hence completes the proof of (b).
\end{proof}

Note that in Theorem~\ref{psqp}(a), we assume that $\Th=\R^n$ in the composite problem \eqref{comp}.
The reason for this assumption is that the proof requires that  the vector $H_k(x_{k+1}-x_k)$ in \eqref{fn88} converge  to $0$. 
As shown in the given proof, this can be achieved when $\Th=\R^n$. While this assumption does not seem to be restrictive, it is unclear whether it can be omitted.

We can replace the second-order sufficient condition \eqref{sosc} with the noncriticality 
of Lagrange multipliers in Theorem~\ref{psqp}(a) by replacing the convex cone $\D$ in the Dennis-Mor\'e condition \eqref{dmc}  with the linear subspace $\D_+$.
 
\begin{Theorem}[primal  superlinear convergence under noncriticality]\label{psqp2}
 Let $(\ox,\olm)$ be a solution to the KKT system \eqref{vs} with $\Th=\R^n$, let  $\{H_k\b$ be a sequence of  $n\times n$ symmetric matrices 
 and let $\{(x_k,\lm_k)\b$ be constructed via Algorithm~{\rm\ref{sqpal}}.
 Assume further that the sequence $\{(x_k,\lm_k)\b$ converges to  $(\ox,\olm)$ as $k\to \infty$.  
If $\olm$ is a noncritical Lagrange multiplier for the KKT system \eqref{vs} and if the condition 
\begin{equation}\label{dmc2}
P_{\D_+}\big(\big( \nabla_{xx}^2L(x_k,\lm_k)- H_k\big)(x_{k+1}-x_k)\big)=o(\|x_{k+1}-x_k\|)
\end{equation}
with $\D_+$ taken from \eqref{cones} is satisfied, then the rate of convergence of the primal sequence $\{x_k\b$ is superlinear.
\end{Theorem}
\begin{proof} Define $p_{k+1}$ and $v_{k+1}$ as the beginning of the   proof of Theorem~\ref{psqp} and observe that \eqref{fn88}-\eqref{pe20}
hold.  For some neighborhoods
$U$ of $(0,0)\in \R^n\times \R^m$ and $V$ of $(\ox,\olm)$, obtained in  Theorem~\ref{prim2}, we can show via \eqref{pe20} as the proof of Theorem~\ref{psqp}(a) that 
$(v_{k+1},p_{k+1})\in U$  and  $( x_{k+1},\lm_{k+1})\in S( v_{k+1},p_{k+1})\cap V$ for all $k\in \N$ sufficiently large. Appealing now to \eqref{semi3} and   \eqref{dmc2} yields the estimates 
\begin{eqnarray*}
\|x_{k+1}-\ox\| &\le & \kappa\big(\|P_\D{_+}(v_{k+1})\|+\|p_{k+1}\|\big)\\
&\le & \kappa\big(\|P_\D{_+}(u_{k+1})\|+ o(\|x_{k+1}-x_k\|)\big)= o(\|x_{k+1}-x_k\|). 
\end{eqnarray*}
Following a similar argument as the proof of  \eqref{fn90} justifies the superlinear convergence of   the primal sequence $\{x_k\b$. 
\end{proof}
\begin{Remark}[discussion on primal superlinear convergence] {\rm 
Below, we discuss several issues related to Theorems~\ref{psqp} and \ref{psqp2}:

\begin{itemize}[noitemsep,topsep=0pt] 

\item[\rm{(a)}]  The primal superlinear convergence of the quasi-Newton SQP method  was studied 
for NLPs with only equality constraints in \cite{btw} and with both equality and inequality constraints in \cite{bo94}. 
The latter was slightly improved in \cite[Theorem~15.7]{bgls} in which it was justified for NLPs that under the second-order sufficient condition and the linear independence 
constraint qualification, the primal superlinear convergence of  the quasi-Newton SQP methods   amounts to a counterpart of the Dennis-Mor\'e condition \eqref{dmc} for this setting. 
 This result  was significantly improved in \cite[Theorem~4.1]{fis12} by showing that 
 the second-order sufficient condition alone suffices to establish   the latter characterization of the primal superlinear  convergence of quasi-Newton SQP methods. 
  Theorem~\ref{psqp} extends this characterization  for the composite problem \eqref{comp}. 

One can also find similar   results for generalized equations 
in \cite[Theorem~6E.3]{dr}. There are, however,  three important differences between Theorems~\ref{psqp} and \ref{psqp2} and those in \cite[Chapter~6]{dr}
applied to the generalized equation \eqref{ge2}. 
First,  \cite[Theorem~6E.3]{dr} utilizes the isloated calmness of the solution mapping $S$ from \eqref{maps}, which 
is strictly stronger than the noncriticality assumption exploited in Theorem~\ref{psqp2}.  In fact, while the former requires the uniqueness of Lagrange multipliers, 
the latter does not demand such a restriction on the Lagrange multiplier set $\Lm(\ox)$. Second,  
the imposed Dennis-Mor\'e condition in \cite[Theorem~6E.3]{dr} can be formulated for  \eqref{ge2} as   
\begin{equation}\label{dmc3}
\big( \nabla^2_{xx}L( x_k,\lm_k)-H_k\big)(x_{k+1}-x_k)=o(\|x_{k+1}-x_k\|),
\end{equation}
which clearly implies the Dennis-Mor\'e condition \eqref{dmc}. 
Finally, \cite[Theorem~6E.3]{dr} provides necessary and sufficient conditions -- not a characterization -- for the primal-dual superlinear convergence of the quasi-Newton SQP method.
In contrast,  Theorem~\ref{psqp}  achieves a characterization of the primal superlinear convergence of the latter method without requiring the uniqueness of Lagrange multipliers, assumed in \cite{dr}.
Note  that results as \cite[Theorem~6E.3]{dr} for the composite problem \eqref{comp} can be derived using Theorem~\ref{chcri} without   assuming the uniqueness of Lagrange multipliers
  and do not require the sharper primal estimates that were established in Theorems~\ref{prim1} and \ref{prim2}.

\item[\rm{(b)}] For the basic SQP method, namely when the matrices $H_k$ are chosen as \eqref{hk}, it follows from Theorem~\ref{psqp}
 that the second-order sufficient condition \eqref{sosc} and $\Lm(\ox)=\{\olm\}$, being equivalent to the dual condition \eqref{dual} by Proposition~\ref{unic},
 ensures the existence of a primal-dual sequence $\{(x_k,\lm_k)\b$ that converges to $(\ox,\olm)$. Observe also that   the Dennis-Mor\'e condition \eqref{dmc} automatically holds for this choice of $H_k$.
Combining these  ensures  the primal superlinear convergence of the basic SQP method for \eqref{comp} with $\Th=\R^n$ under these two assumptions.
\end{itemize}
}
\end{Remark}

Note that \cite[Theorem~6E.3]{dr} was used recently by Burke and Engle in \cite[Theorem~5.2]{be} for the composite problem \eqref{comp} with $\Th=\R^n$ 
to derive a primal-dual superlinear convergence of the quasi-Newton SQP method by assuming the condition \eqref{dmc3} and the isolated calmness of 
the solution mapping $S$ from \eqref{maps}--  the authors in \cite{be} assumed the strong metric subregularity of the mapping $G$ from \eqref{GKKT},
which is equivalent to the isolated calmness of $S$ since $S=G^{-1}$. Below we show that the later condition can be 
weakened to the noncriticality assumption, a condition that  does not necessarily yield the uniqueness of Lagrange multipliers.
\begin{Theorem}[primal  superlinear convergence under noncriticality]\label{psqp3}
 Let $(\ox,\olm)$ be a solution to the KKT system \eqref{vs}, let  $\{H_k\b$ be a sequence of  $n\times n$ symmetric matrices 
 and let $\{(x_k,\lm_k)\b$ be constructed via Algorithm~{\rm\ref{sqpal}}.
 Assume further that the sequence $\{(x_k,\lm_k)\b$ converges to  $(\ox,\olm)$ as $k\to \infty$.  
If $\olm$ is a noncritical Lagrange multiplier for the KKT system \eqref{vs} and if \eqref{dmc3} is satisfied, then the rate of convergence of the primal sequence $\{x_k\b$ is superlinear.
\end{Theorem}
\begin{proof} Define $p_{k+1}$, $u_{k+1}$, and $v_{k+1}$ as the beginning of   the proof of Theorem~\ref{psqp} and observe that \eqref{fn88}-\eqref{pe20}
hold.   Since $\olm$ is noncritical, it follows from  Theorem~\ref{chcri} that there are  neighborhoods
$U$ of $(0,0)\in \R^n\times \R^m$ and $V$ of $(\ox,\olm)$  and a constant $\kappa\ge 0$ such that 
for any $(v,p)\in U$ and any  $(x,\lm)\in  S(v,p)\cap V$ the estimate \eqref{semi} holds, where $S$ comes from \eqref{maps}.
By $(x_k,\lm_k)\to (\ox,\olm)$, the condition \eqref{dmc3},  and the definitions of $p_{k+1}$, $u_{k+1}$, and $v_{k+1}$, we arrive at $p_{k+1}\to 0$, $u_{k+1}\to 0$, and $v_{k+1}\to 0$
as $k\to \infty$. So by \eqref{pe20}, we can assume with no harm that $(v_{k+1},p_{k+1})\in U$  and  $( x_{k+1},\lm_{k+1})\in S( v_{k+1},p_{k+1})\cap V$ for all $k\in \N$. 
 Appealing now to \eqref{semi} and   \eqref{dmc3} yields the estimates 
\begin{eqnarray*}
\|x_{k+1}-\ox\| &\le & \kappa\big(\|v_{k+1}\|+\|p_{k+1}\|\big)\\
&\le & \kappa\big(\| u_{k+1}\|+ \|q_{k+1}\|+ \|p_{k+1}\|\big)= o(\|x_{k+1}-x_k\|),
\end{eqnarray*}
Following a similar argument as the proof of  \eqref{fn90} justifies the superlinear convergence of   the primal sequence $\{x_k\b$. 
\end{proof}

{\bf Acknowledgements.} We thank the two anonymous reviewers whose comments and suggestions helped improve the original presentation of the paper.  



\end{document}